\documentclass[12pt]{amsart}
\usepackage{graphicx}
\newtheorem{theorem}{Theorem}[section]
\newtheorem{proposition}[theorem]{Proposition}
\newtheorem{lemma}[theorem]{Lemma}

\theoremstyle{remark}

\numberwithin{equation}{section}

\begin{document}

\title[Q-fundamental surfaces in lens spaces] {
Q-fundamental surfaces in lens spaces}

\author{Miwa Iwakura and Chuichiro Hayashi}

\date{\today}

\thanks{The last author is partially supported
by Grant-in-Aid for Scientific Research (No. 18540100),
Ministry of Education, Science, Sports and Technology, Japan.}

\begin{abstract}
 We determine all the Q-fundamental surfaces in $(p,1)$-lens spaces
and $(p,2)$-lens spaces 
with respect to natural triangulations with $p$ tetrahedra.
 For general $(p,q)$-lens spaces,
we give an upper bound for elements of vectors 
which represent Q-fundamental surfaces
with no quadrilateral normal disks 
disjoint from the core circles of lens spaces.
 We also give some examples of non-orientable closed surfaces 
which are Q-fundamental surfaces
with such quadrilateral normal disks.
\end{abstract}

\maketitle


\section{Introduction}
 In \cite{K}, K. Kneser introduced normal surfaces
which are in $\lq\lq$beautiful" position
with respect to a triangulation of a $3$-manifold,
to show the existence of prime decomposition.
 A normal surface intersects each tetrahedron of a triangulation
in a disjoint union of trigons and quadrilaterals.

 In \cite{H}, 
W. Haken found that normal surfaces 
correspond to non-negative solutions 
of a certain system of simultaneous linear equations with integer coefficients,
called the matching equations.
 He introduced fundamental surfaces, 
to obtain an algorithm to decide if a given knot is trivial or not.
 The set of fundamental solutions forms the Hilbert basis
of the space of non-negative integer solutions.
 There are only finite number of fundamental solutions for each system.
 Since almost all sorts of important surfaces,
such as essential spheres, essential tori, essential disks, essential annuli, 
knot spanning surfaces with minimal Euler characteristics and so on, 
can be deformed to fundamental surfaces, 
we obtain principal algorithms for $3$-manifold topology.
 See, for example, \cite{JT} and \cite{S}.
 The theory of fundamental surfaces was also applied to problems 
other than algorithms.
 See \cite{HL} and \cite{L}, for example.

 There is an algorithm
which determines all the fundamental surfaces
for a given triangulation of a $3$-manifold.
 However, it is not practical 
because the numbers of variables and equations
of the matching equations are very large
even for simplest $3$-manifolds.
 Variables of the matching equations 
correspond to types of normal disks.
 There are $4$ types of trigonal normal disks 
and $3$ types of quadrilateral normal disks
in each tetrahedron as shown in Figure \ref{fig:normaldisks}.
 Hence the number of variables is $7t$
if a triangulation has $t$ tetrahedra.
 When a closed $3$-manifold with a triangulation $T$ is given,
three matching equations arise for every face of $T$, 
and hence the number of equations is $3 \times 4t/2 =6t$.
 The Q-theory introduced by J. Tollefson in \cite{T}
uses the Q-matching equations
with only $3t$ variables corresponding to the types of quadrilaterals.
 The number of equations is equal to the number of edges of the triangulation.
 However, Q-theory does not take care of Euler characteristic. 
 Incompressible surfaces can be deformed into Q-fundamental surfaces.
 But, the resulting surfaces may have smaller Euler characteristics 
than the original ones.

 In practice, normal surfaces
in the figure eight knot complement with an ideal triangulation
are studied very well.
 See \cite{AMR}, \cite{Ka}, \cite{MR} and \cite{R}.
 In \cite{F},
E. A. Fominykh gives a complete description of the fundamental surfaces
for infinite series of $3$-manifolds including $(p,q)$-lens spaces
with certain handle decompositions (not with triangulations).
 His argument is geometrical.
 Haken's original normal surface theory is based on handle decompositions
rather than triangulations.
 However, triangulation is easier to be treated by a computer.

 In this paper, 
we determine all the Q-fundamental surfaces 
for $(p,1)$-lens spaces and $(p,2)$-lens spaces
naturally triangulated with $p$ tetrahedra
by algebraic calculations.
 For general $(p,q)$-lens spaces,
we give an upper bound for elements of vectors 
which represent Q-fundamental surfaces
with no quadrilateral normal disks 
disjoint from the core circles of lens spaces.
 On the other hand,
in the last section
we also give an example of non-orientable closed surface 
of maximal Euler characteristic
which is fundamental
and has three parallel copies of such quadrilateral normal disks.
 In fact, 
there is an infinite sequence of $(p_n, q_n)$-lens spaces ($n \ge 2$)
which contain fundamental surfaces
homeomorphic to the connected sum of $n$ projective planes
with $n-2$ parallel sheets of a normal disk as above
(the detail will be given in \cite{IH}).
 In lens spaces, 
non-orientable closed surfaces with maximum Euler characteristics 
are interesting.
 A formula for calculating the maximum Euler characteristic
is given by G. E. Bredon and J. W. Wood in \cite{BW}.

 B. A. Burton established a computer program {\it Regina}
which determines all the $\lq\lq$vertex surfaces"
with respect to the (Q-)matching equations
for a given triangulated $3$-manifold.
 See \cite{B}.
 The authors don't know
whether there is a computer software 
which determines all the fundamental surfaces
for general triangulated $3$-manifolds or not.
 Vertex surfaces are introduced in \cite{JT} by W. Jaco and J. Tollefson.
 We can calculate them by computer much more easily
than fundamental ones.
 Most important two-sided surfaces can be deformed into vertex surfaces.
 However, non-orientable surfaces in lens spaces are one-sided.
 In fact, we 
see later in Theorem \ref{theorem:main}
that the non-orientable closed surface with maximum Euler characteristic
in the $(p,1)$-lens space can be represented by a Q-fundamental surface,
and cannot be by a normal surface corresponding to a vertex solution.



\begin{figure}[htbp]
\begin{center}
\includegraphics[width=6cm]{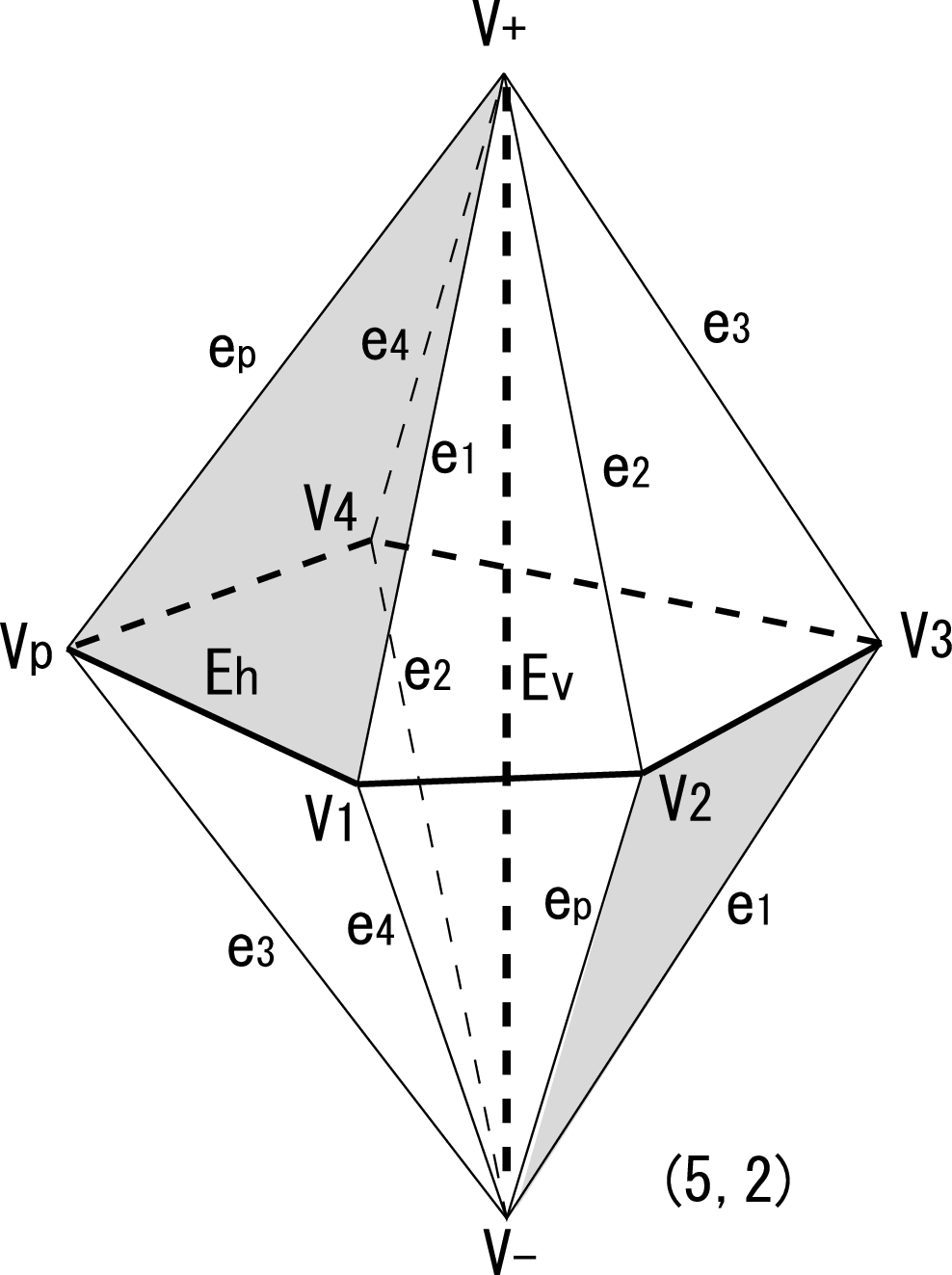}
\end{center}
\caption{}
\label{fig:suspension}
\end{figure}

 We can obtain a {\it $(p,q)$-lens space} from a suspension of a $p$-gon 
by gluing each trigonal face in the upper hemisphere with that 
in the lower hemisphere, performing ($2\pi q/p$)-rotation 
and taking a mirror image about the equator. 
 See Figure \ref{fig:suspension}.
 Precisely, the trigon $v_+ v_i v_{i+1}$ is glued to $v_- v_{i+q} v_{i+q+1}$, 
where indices are considered in modulo $p$. 
 The edge $e_i$ connects $v_+$ and $v_i$, and also $v_-$ and $v_{i+q}$. 
 The horizontal edges are all glued up together into an edge $E_h$. 
 Taking an axis $E_v$ connecting the vertices $v_+$ 
and $v_-$ in the suspension, we can decompose it into $p$ tetrahedra. 
 This gives a natural triangulation $T(p,q)$ of a $(p,q)$-lens space. 
 The $i$-th tetrahedron $\tau_i$ has 
vertices $v_+$, $v_-$, $v_i$ and $v_{i+1}$. 

\begin{figure}[htbp]
\begin{center}
\includegraphics[width=8cm]{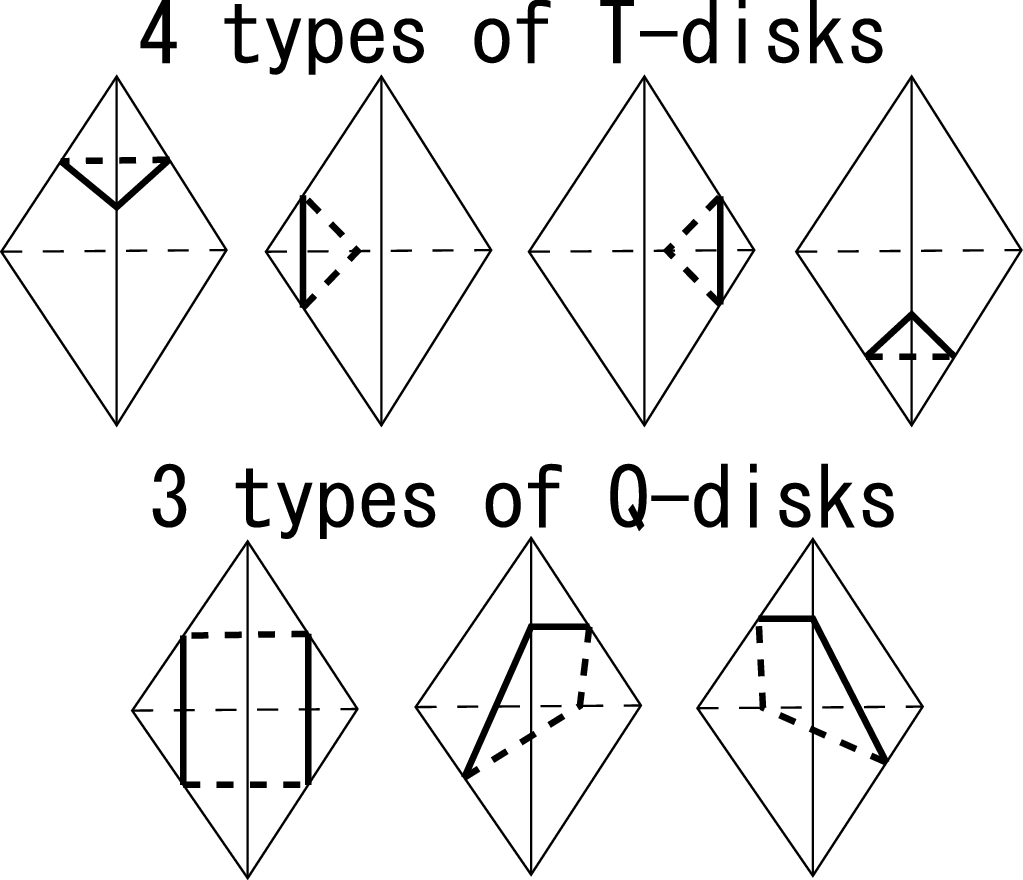}
\end{center}
\caption{}
\label{fig:normaldisks}
\end{figure}

 We recall the definition of Q-fundamental surfaces.
 Let $M$ be a closed $3$-manifold and $T$ a triangulation of $M$, that is, 
a decomposition of $M$ into finitely many tetrahedra. 
 Let $F$ be a closed surface embedded in $M$. 
 $F$ is called a {\it normal surface} with respect to the triangulation $T$ 
if $F$ intersects each tetrahedron $\tau$ in disjoint union of normal disks, 
or in the empty set.
 There are two kinds of {\it normal disks}, 
trigons called T-disks, and quadrilaterals called Q-disks. 
 Each tetrahedron contains $4$ types of T-disks and $3$ types of Q-disks 
as in illustrated in Figure \ref{fig:normaldisks}. 

\begin{figure}[htbp]
\begin{center}
\includegraphics[width=8cm]{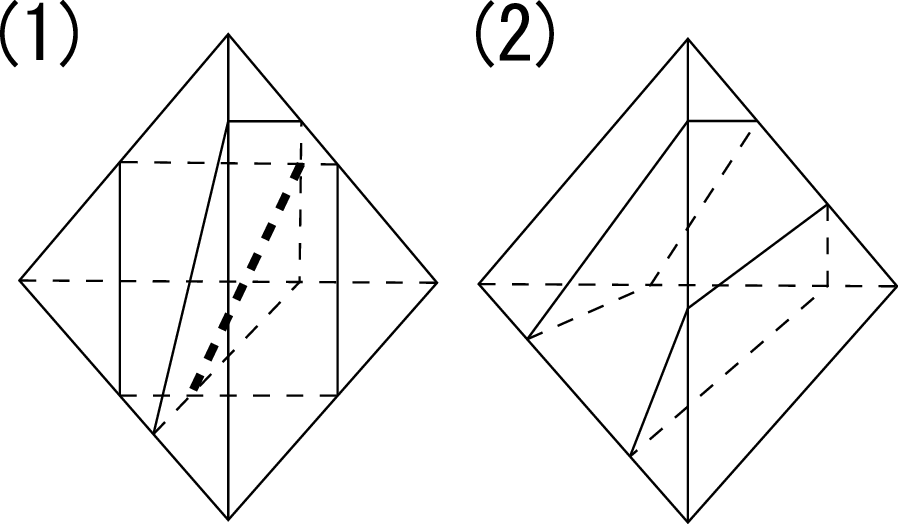}
\end{center}
\caption{}
\label{fig:squarecondition}
\end{figure}

 If two or three types of Q-disks exist in a single tetrahedron, 
then they intersect each other. 
 See Figure \ref{fig:squarecondition} (1).
 Hence any normal surface intersects 
each tetrahedron in Q-disks of the same type and T-disks
(Figure \ref{fig:squarecondition} (2)). 
 This is called {\it  the square condition}. 

\begin{figure}[htbp]
\begin{center}
\includegraphics[width=10cm]{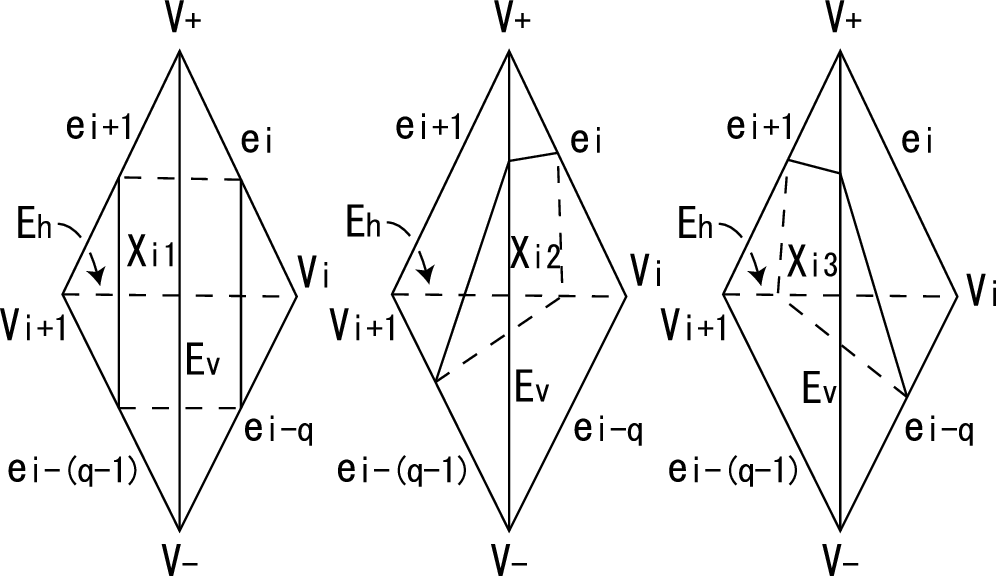}
\end{center}
\caption{}
\label{fig:numberQdisk}
\end{figure}

 Tollefson introduced Q-coordinates representing normal surfaces
in \cite{T}. 
 We first number types of Q-disks. 
 In our case, 
we number the $3$-types of Q-disks $i1$, $i2$ and $i3$ in $\tau_i$
as in Figure \ref{fig:numberQdisk}, 
where the axis $E_v$ is in front of the tetrahedron. 
 $X_{i1}$ separates $E_h$ and $E_v$,
$X_{i2}$ separates $e_{i+1}$ and $e_{i-q}$
and $X_{i3}$ does $e_i$ and $e_{i-(q-1)}$.
 Let $x_{ij}$ be the number of Q-disks of type $ij$
contained in a normal surface $F$. 
 Then we place them in a vertical line 
in the order of indices lexicographically, 
to obtain a {\it Q-coordinate} of $F$. 

 Tollefson found the Q-coordinate of a normal surface 
satisfies the linear system of the Q-matching equations as below.
 One Q-matching equation arises 
from each edge of the triangulation $T$ under consideration.
 Let $e$ be an edge of $T$.
 We observe each tetrahedron containing $e$, placing $e$ in front. 
 Then the sense of a Q-disk with respect to $e$ is $0, +1$ or $-1$ 
if it is flat, left side up or right side up respectively. 
 See Figure \ref{fig:sense}.
 The Q-matching equation about $e$ derives 
from the constraint that the number of Q-disks with plus sense is equal 
to that with minus sense around the edge $e$. 
 See Figure \ref{fig:constraint}.

\begin{figure}[htbp]
\begin{center}
\includegraphics[width=7cm]{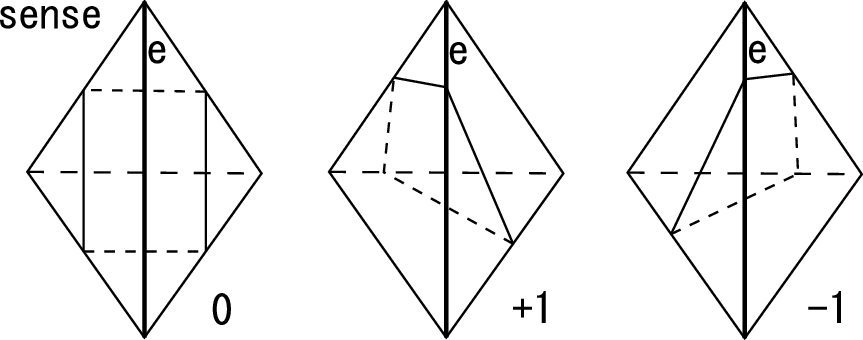}
\end{center}
\caption{}
\label{fig:sense}
\end{figure}

\begin{figure}[htbp]
\begin{center}
\includegraphics[width=5cm]{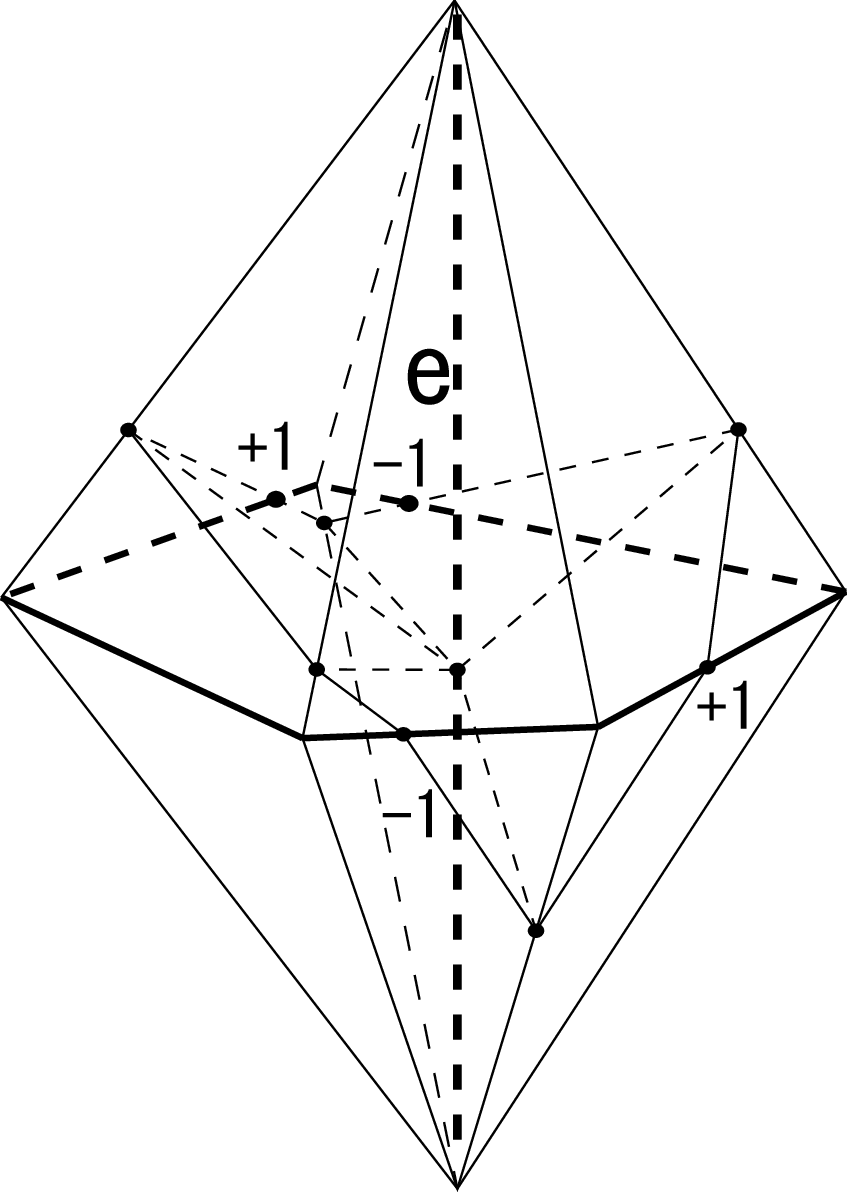}
\end{center}
\caption{}
\label{fig:constraint}
\end{figure}

 To define Q-fundamental surface, we need some terminologies on algebra. 
 Let 
${\bf v} = {}^t (v_1, \cdots, w_n)$, ${\bf w} = {}^t (w_1, \cdots, w_n)$ 
be vectors in ${\Bbb R}^n$,
where ${}^t{\bf x}$ denotes the transposition of ${\bf x}$.
 We write ${\bf v} \le {\bf w}$ 
if $v_i \le w_i$ for all $i \in \{ 1, \cdots, n \}$.
 ${\bf v} < {\bf w}$
means that both ${\bf v} \le {\bf w}$ and ${\bf v} \ne {\bf w}$ hold.
 A vector ${\bf u} \in {\Bbb R}^n$ is {\it non-negative} 
if ${\bf 0} \le {\bf u}$,
and {\it integral} if all its elements are in ${\Bbb Z}$.

 We consider {\it Q-matching equations} for all the edges, 
to obtain a linear system of equations. 
 Tollefson showed 
that a non-zero non-negative integral solution of the Q-matching equations 
determines a unique normal surface with no trivial component
if it satisfies the square condition, 
where a trivial component is a normal surface 
composed of T-disks and containing no Q-disks, 
that is, a $2$-sphere surrounding a vertex of $T$. 

 Let $A {\bf x}={\bf 0}$ be a linear system of equations, 
where $A$ is a matrix with all the elements in ${\Bbb Z}$, 
and ${\bf x}$ is a vector of variables. 
 $V_A$ denotes the solution space of the linear system 
considered in ${\Bbb R}^n$. 
 A non-zero non-negative integral solution ${\bf v}$ 
is called a {\it fundamental solution}, 
if there are no integral solution ${\bf v}' \in V_A$ 
with ${\bf 0} < {\bf v}' < {\bf v}$.
 A non-zero non-negative integral solution ${\bf v}$ 
is called a {\it vertex solution}, 
if for every positive integer $k$, 
all the integral solutions ${\bf v}''$ 
with ${\bf 0} \le {\bf v}'' \le k{\bf v}$ 
are multiples of ${\bf v}$. 
 This condition coincides with that for fundamental solutions when $k=1$. 
 Hence a vertex solution is a fundamental solution. 

 A normal surface is called 
a {\it Q-fundamental surface} 
if it corresponds to a fundamental solution
of the system of the Q-matching equations respectively. 
 A two sided normal surface is called 
a {\it vertex surface}
if its coordinate is either a vertex solution of the Q-maching equations
or twice a vertex solution representing a one sided surface.
 Tollefson showed that 
if an irreducible $\partial$-irreducible triangulated $3$-manifold 
contains a two sided incompressible $\partial$-incompressible surface, 
then there is such one which is Q-vertex. 
 Note that the set of fundamental surfaces 
with respect to Haken's matching equations 
contains such a surface with maximum Euler characteristic, 
while the set of the Q-fundamental surfaces may not.  


 For non-orientable closed surfaces with maximal Euler characteristics
in lens spaces, 
we can apply the next theorem,
since non-orientable surfaces in orientable $3$-manifolds are 
one sided, and hence non-separating.
 Note that a non-separating surface in a $(p,q)$-lens space $M$ 
must be non-orientable
because the $1$-dimensional homology 
$H_1 (M; {\Bbb Z}) = {\Bbb Z} / p{\Bbb Z}$.

\begin{theorem}\label{theorem:NonSepFund}
 Let $M$ be a closed $3$-manifold,
$T$ a triangulation of $M$.
 Suppose that $M$ contains a non-separating closed surface $F$.
 Then $M$ contains such one which is Q-fundamental,
and such one with maximum Euler characteristic which is fundamental
with respect to $T$.
\end{theorem}

 In the above theorem, $M$ is orientable or non-orientable,
and $F$ is one sided or two sided.
 To establish a result of this type on one sided surface
in order to obtain a one sided fundamental surface
seems to be almost impossible
since a Haken sum on two sided surfaces may yield a one sided surface
as in Figure \ref{fig:CutAndPaste}.

\begin{figure}[htbp]
\begin{center}
\includegraphics[width=8cm]{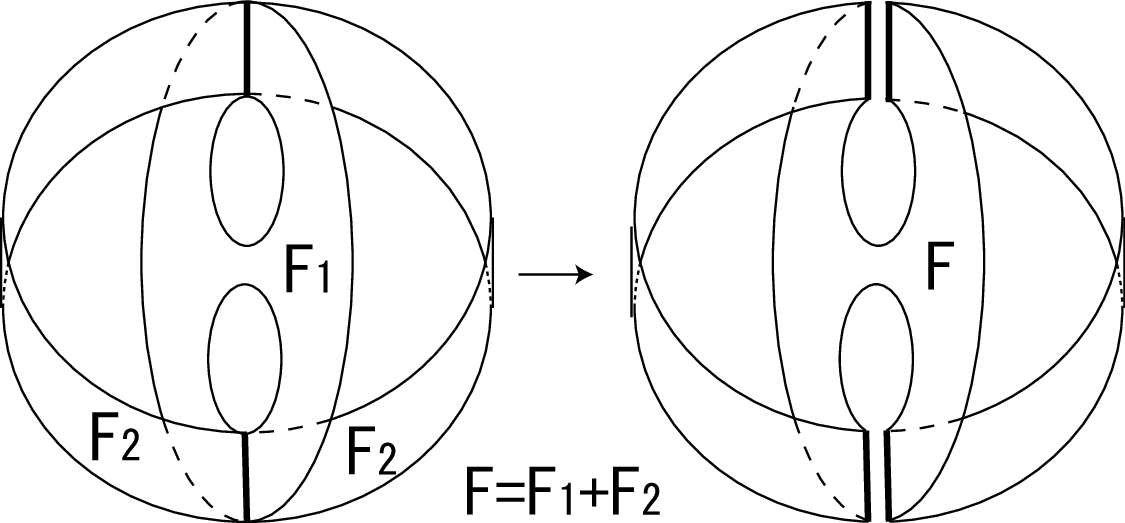}
\end{center}
\caption{}
\label{fig:CutAndPaste}
\end{figure}

 In our situation of $(p,q)$-lens space, 
the $(3i-2)$-nd, $(3i-1)$-st and $3i$-th elements of a Q-coordinate 
together form the $i$-th {\it block} for $1\le i \le p$. 
 These elements correspond to the numbers of quadrilateral disks 
of type $X_{i1}, X_{i2}, X_{i3}$ in $\tau_i$.
 We often put a vertical line $\lq\lq |$" instead of a comma
between every adjacent pair of blocks.

\begin{theorem}\label{theorem:main}
\begin{enumerate}
\item[(i)]
 The triangulation $T(2,1)$ of $(2,1)$-lens space has
exactly three Q-fundamental surfaces
${\bf f}_1={}^t(1,0,0\ |\ 1,0,0)$,
${\bf f}_2={}^t(0,1,0\ |\ 0,0,1)$,
${\bf f}_3={}^t(0,0,1\ |\ 0,1,0)$.

 ${\bf f}_1$ represents a Heegaard splitting torus 
which surrounds the core circle $E_v$. 
 It does $E_h$ also on the other side of it.
 Each of ${\bf f}_2$ and ${\bf f}_3$ represents a projective plane.
 $\pi$-rotation about the axis $E_v$ 
carries ${\bf f}_3$ to ${\bf f}_2$. 
 $2{\bf f}_2$ represents the inessential $2$-sphere surrounding the edge $e_2$,
and $2{\bf f}_3$ such one surrounding $e_1$.
\item[(ii)]
 The triangulation $T(p,1)$ of $(p,1)$-lens space with $p \ge 3$ has 
exactly $p+3$ Q-fundamental surfaces 
${\bf t}'_1, \cdots, {\bf t}'_p, {\bf f}'_1, {\bf f}'_2$ and ${\bf f}'_3$ below 
when $p$ is even,
and $p+1$ Q-fundamental surfaces 
${\bf t}'_i$'s and ${\bf f}'_1$ when $p$ is odd. 
\newline
${\bf t}'_i
={}^t(0,0,0\ | \cdots |\ 0,0,0\ |\ 0,1,0\ |\ 0,0,2\ |\ 0,1,0\ |\ 0,0,0\ | 
\cdots |\ 0,0,0)$
\newline
${\bf f}'_1={}^t(1,0,0\ |\ 1,0,0\ | \cdots |\ 1,0,0)$
\newline
${\bf f}'_2
={}^t(0,1,0\ |\ 0,0,1\ |\ 0,1,0\ |\ 0,0,1\ | \cdots |\ 0,1,0\ |\ 0,0,1)$
\newline
${\bf f}'_3
={}^t(0,0,1\ |\ 0,1,0\ |\ 0,0,1\ |\ 0,1,0\ | \cdots |\ 0,0,1\ |\ 0,1,0)$
\newline
 ${\bf t}'_i$ has ${}^t(0\ 0\ 2)$ as the $i$-th block,
and represents the inessential $2$-sphere surrounding the edge $e_i$.
 ${\bf f}'_2$ and ${\bf f}'_3$ represent a non-orientable closed surface 
with maximum Euler characteristic, 
which is the connected sum of $p/2$ projective planes. 
 Note that $2\pi/p$-rotation about the axis $E_v$ 
carries ${\bf f}'_3$ to ${\bf f}'_2$. 
 ${\bf f}'_2$ and ${\bf f}'_3$ are not 
vertex solutions of the Q-mathing equations
since $2{\bf f}'_2 = {\bf t}'_2 + {\bf t}'_4 + \cdots + {\bf t}'_p$ and
$2{\bf f}'_3 = {\bf t}'_1 + {\bf t}'_3 + \cdots + {\bf t}'_{p-1}$.
 ${\bf f}'_1$ represents a Heegaard splitting torus 
which surrounds the core circles $E_v$ and $E_h$. 
\item[(iii)]
 Let $p$ be an odd integer with $p\ge 5$. 
 The triangulation $T(p,2)$ of $(p,2)$-lens space has 
exactly $p+1$ Q-fundamental surfaces 
\newline
${\bf t}''_i={}^t(0,0,0\ | \cdots |\ 0,0,0\ |\ 0,1,0\ |\ 0,0,1\ |\ 0,0,1\ |\ 0,1,0\ |\ 0,0,0\ | \cdots |\ 0,0,0)$,
\newline
${\bf f}''_1={}^t(1,0,0\ |\ 1,0,0\ | \cdots |\ 1,0,0)$.
\newline
 ${\bf t}''_i$ 
with the $(i-1)$-th and the $(i+2)$-th blocks being ${}^t(0\ 1\ 0)$
represents the $2$-sphere surrounding the edge $e_i$.
 ${\bf f}''_1$ represents the Heegaard splitting torus 
surrounding $E_v$ and $E_h$. 
\end{enumerate}
\end{theorem}


 We consider a general $(p,q)$-lens space.
 Since the $(p,q)$-lens space is homeomorphic 
to the $(p,p-q)$-lens space,
it is sufficient to consider the case of $2 \le q < p/2$.

\begin{lemma}\label{lemma:generator}
 For the triangulation $T(p,q)$ of the $(p,q)$-lens space
with $p \ge 5$, $2 \le q < p/2$ and $GCM(p,q)=1$, 
the vectors
${\bf s}_1, {\bf s}_2, \cdots, {\bf s}_p$,
${\bf t}_1, {\bf t}_2, \cdots, {\bf t}_p$ as below
form a basis of the solution space in ${\Bbb R}^{3t}$
of the Q-matching equations.
\newline
 $($The $j$-th block of ${\bf s}_i) = \left\{ 
\begin{array}{l}
{}^t(1,1,1)\ \ {\rm if}\ j=i \\
{}^t(0,0,0)\ \ {\rm otherwise}
\end{array} \right.$
\newline
 $($The $j$-th block of ${\bf t}_i)= \left\{ 
\begin{array}{l}
{}^t(0,1,0)\ \ {\rm if}\ j=i-1\ {\rm or}\ i+q \\
{}^t(0,0,1)\ \ {\rm if}\ j=i\ {\rm or}\ i+q-1 \\
{}^t(0,0,0)\ \ {\rm otherwise}
\end{array} \right.$

 Hence a general solution ${\bf v}$ in ${\Bbb R}^{3t}$
is presented as below. 
\newline
${\bf v}=a_1 {\bf s}_1 + \cdots +a_p {\bf s}_p 
+ b_1 {\bf t}_1 + \cdots + b_p {\bf t}_p$
\newline
$= 
{}^t (a_1,\ a_1+b_2+b_{p-q+1},\ a_1+b_1+b_{p-q+2} \ | \cdots |\ 
a_i,\ a_i+b_{i+1}+b_{p-q+i},\ a_i+b_i+b_{p-q+i+1} \ | \cdots $
\newline
\hspace*{2cm}
$\cdots |\ a_p,\ a_p+b_1+b_{p-q},\ a_p+b_p+b_{p-q+1})$,
\newline
with $a_1, \cdots, a_p, b_1, \cdots, b_p\in {\Bbb R}$.
\end{lemma}

We set ${\mathcal B}=\{ b_1,b_2,\cdots ,b_p \}$, 
and also 
${\mathcal B}_0=\{ b_2,b_4,\cdots ,b_p \}$, 
${\mathcal B}_1=\{ b_1,b_3,\cdots ,b_{p-1} \}$ 
when $p$ is even. 

\begin{lemma}\label{lemma:p-qZ}
 Suppose that
${\bf v}=
a_1 {\bf s}_1 + \cdots +a_p {\bf s}_p + 
b_1 {\bf t}_1 + \cdots + b_p {\bf t}_p$ 
is contained in ${\Bbb Z}^{3p}$.
 Then $a_k \subset {\Bbb Z}$ for all integer $k$ with $1 \le k \le p$,
and either ${\mathcal B}_x \subset {\Bbb Z}$ or 
${\mathcal B}_x \subset {\Bbb Z}+1/2$ holds
for ${\mathcal B}_x = {\mathcal B}$, ${\mathcal B}_0$ and ${\mathcal B}_1$.
\end{lemma}

\begin{theorem}\label{theorem:p-qfund}
 Assume that 
${\bf v}=
a_1 {\bf s}_1 + \cdots +a_p {\bf s}_p + 
b_1 {\bf t}_1 + \cdots + b_p {\bf t}_p$ 
$(a_1, \cdots, a_p, b_1, \cdots, b_p\in {\Bbb R})$ 
represents a Q-fundamental solution, 
and that $a_k=0$ for all $k$. 
\begin{enumerate}
\item[(I)]
 When $p$ is even, 
one of ${\mathcal B}_0$ and ${\mathcal B}_1$ 
is contained in ${\Bbb Z}$ or ${\Bbb Z}+1/2$
and the other is $\{ 0\}$. 
\item[(II)]
 Suppose 
either $p$ is odd and ${\mathcal B} \subset {\Bbb Z}+1/2$, 
or $p$ is even and one of ${\mathcal B}_0$ and ${\mathcal B}_1$ 
is contained in ${\Bbb Z}+1/2$ and the other is $\{ 0\}$. 
 Then $b_i \in \{-1/2,0,1/2 \}$ for $\forall i \in \{1,\cdots ,p\}$. 
\item[(III)]
 Suppose 
either $p$ is odd and ${\mathcal B} \subset {\Bbb Z}$, 
or $p$ is even and one of ${\mathcal B}_0$ and ${\mathcal B}_1$ 
is contained in ${\Bbb Z}$ and other is $\{ 0\}$D
 Then $b_i \in \{-1,0,1 \}$ for $\forall i \in \{1,\cdots ,p\}$. 
\end{enumerate}
\end{theorem}

\begin{theorem}\label{theorem:p-qfund-podd1/2}
 If $p$ is odd, $q\ge 2$, ${\mathcal B} \subset {\Bbb Z}+1/2$
and $a_k = 0$ for $\forall k \in \{ 1,\cdots, p \}$, 
then ${\bf v}=
a_1 {\bf s}_1 + \cdots +a_p {\bf s}_p + 
b_1 {\bf t}_1 + \cdots + b_p {\bf t}_p$ 
cannot be a Q-fundamental solution. 
\end{theorem}

\begin{lemma}\label{lemma:HakenFund}
 Let $F$ be a normal surface 
in the $(p,q)$-lens space
with the triangulation $T(p,q)$.
 If $F$ intersects each of $E_v$ and $E_h$ in a single point
and contains a normal disk of type $X_{k2}$ or $X_{k3}$,
then $F$ is fundamental with respect to Haken's matching equations.
\end{lemma}

 When $p$ is even and $q \ge 3$,
in the $(p,q)$-lens space
the non-orientable closed surface 
represented by the normal surface 
${\bf h} = (\sum_{k=1}^{p/2} {\bf t}_{2k-1})/2 = 
{}^t (0 \ 0 \ 1 | 0 \ 1 \ 0 | 0 \ 0 \ 1 | 0 \ 1 \ 0 |
\cdots | 0 \ 0 \ 1 | 0 \ 1 \ 0 )$ is fundamental by Lemma \ref{lemma:HakenFund}.
 However, it is not Q-fundamental 
because it is larger than ${\bf t}_i$ in Lemma \ref{lemma:generator}.
 In fact, it is not of maximal Euler characteristic.
 We see examples of compressing disks for the surface represented by ${\bf h}$
in the last section.

 The condition $a_k=0$ for $\forall k \in \{ 1, \cdots, p \}$ is very strong.
 Q-fundamental surfaces with some $a_k$ being non-zero exist.
 We show several examples of such Q-fundamental surfaces in the last section.
 Some of them are
non-orientable closed surfaces with maximal Euler characteristics,
and one of them has three parallel sheets of a normal disk of type $X_{k1}$.

\vspace{5mm}

\noindent
{\it Acknowledgement}:
 The authors would like to thank Reiko Hirano
for giving us the data of calculation of fundamental surfaces
for $T(2,1)$ and $T(3,1)$, 
and Nobue Mizubayashi
for data of Q-fundamental surfaces for $T(3,1)$.
 They were seniors of the first author at Japan Women's Univ.


\section{Proof of Theorem \ref{theorem:NonSepFund}}

 Let $F$ be a (possibly disconnected) surface in a $3$-manifold $M$.
 We say $F$ {\it separates} $M$ or $F$ is {\it separating} in $M$
if there are submanifolds $M_+$ and $M_-$ of $M$
such that (1) both $M_+$ and $M_-$ are unions of components of $M-F$,
(2) $M_+ \cup M_- = M-F$, (3) $M_+ \cap M_- = \emptyset$
and (4) $F =$\,closure\,$(M_+) \cap$\,closure\,$(M_-)$.
 If all the components of $F$ are separating in $M$,
then $F$ is separating.
 Hence $F$ contains a component which is non-separating in $M$
when $F$ is non-separating.
 However, the converse is not true,
a disjoint union of two parallel copies of a non-separating two sided surface
is separating.

 Let $M$ be a compact $3$-manifold possibly with boundary,
and $F$ a surface properly embedded in $M$.
 We say $F$ is {\it geometrically compressible}
if there is an embedded disk $D$, called a {\it compressing disk}, in $M$ 
with $D \cap F = \partial D$
such that the circle $\partial D$ does not bound a disk in $F$.
 Otherwise, $F$ is {\it geometrically incompressible}.

 Suppose that a triangulation $T$ of $M$ is given.
 Then, as is well-known, 
a geometrically incompressible surface can be isotoped
so that it is deformed into a normal surface with respect to $T$.
 We can apply the next two lemmas 
to a non-orientable closed surface with maximal Euler characteristic
in a lens space.
 Note that $F$ is one sided or two sided in $M$ 
in the next two lemmas.

\begin{lemma}\label{lemma:incompressible}
 Let $M$ be a compact $3$-manifold (possibly with boundary),
and $F$ a non-separating surface properly embedded in $M$.
 If $F$ is of maximum Euler characteristic 
among all the non-separating surfaces properly embedded in $M$,
then $F$ is geometrically incompressible.
\end{lemma}

\begin{proof}
 Suppose, for a contradiction, 
that $F$ is geometrically compressible.
 We perform a compressing operation on $F$ along a compressing disk $D$.
 That is, we take a tubular neighbourhood $N(D) \cong D \times I$ of $D$
so that $N(D) \cap F = (\partial D) \times I$,
and set $F'=(F-((\partial D) \times I)) \cup (D \times \partial I)$.
 The resulting compressed surface $F'$ 
is of larger Euler characteristic than $F$ by two.
 In fact, $F'$ is non-separating.
 Suppose not.
 Then $F'$ separates $M$ into two regions $M'_+$ and $M'_-$.
 Without loss of generality, 
we can assume that $N(D) \subset$\,closure\,$(M'_+)$.
 Then $F$ separates $M$ into two regions
$M'_+ - N(D)$ and int($N(D) \cup M'_-$).
 This is a contradiction.
\end{proof}

\begin{lemma}\label{lemma:switch}
 Let $M$ be a compact $3$-manifold (possibly with boundary),
and $T$ a triangulation of $M$.
 If $F = F_1 + F_2$
where $F, F_1, F_2$ are (possibly disconnected) surfaces
that are normal with respect to $T$
and $F_1 + F_2$ denotes the Haken sum,
and if $F$ is non-separating in $M$,
then either $F_1$ or $F_2$ is non-separating.
\end{lemma}

\begin{proof}
 Suppose, for a contradiction, that 
$F_i$ separates $M$ into two submanifolds $M_i^+$ and $M_i^-$
for $i=1$ and $2$.
 Each switch along an intersection curve of $F_1 \cap F_2$
joins either $M_1^+ \cap M_2^+$ and $M_1^- \cap M_2^-$
or $M_1^+ \cap M_2^-$ and $M_1^- \cap M_2^+$.
 Hence $F$ separates $M$ into the two regions
$(M_1^+ \cap M_2^+) \cup (M_1^- \cap M_2^-)$
and $(M_1^+ \cap M_2^-) \cup (M_1^- \cap M_2^+)$.
 This is a contradiction.
\end{proof}

 The next Propositions \ref{proposition:Fund} and \ref{proposition:Q-Fund}
show Theorem \ref{theorem:NonSepFund}.

\begin{proposition}\label{proposition:Fund}
 Let $M$ be a closed $3$-manifold with a triangulation $T$.
 If $M$ contains a non-separating closed surface $F$,
then it contains a non-separating closed surface $F'$ 
such that $F'$ is fundamental with respect to $T$
and has Euler characteristic $\chi(F') \ge \chi(F)$.
\end{proposition}

\begin{proof}
 We use Lemma 2.1 in \cite{JO}.
 See section 2 of \cite{JO}
for definitions of terminologies 
such as reduced form, patches and so on.
 Though handle decompositions of $3$-manifolds are considered in \cite{JO},
the arguments in the proof of Lemma 2.1 are valid also for triangulations.

 Let $F''$ be a non-separating surface in $M$
such that it has maximal Euler characteristic
among all the non-separating surfaces in $M$.
 Then $F''$ is geometrically incompressible by Lemma \ref{lemma:incompressible}.
 Hence an adequate isotopy deforms $F''$ 
to a normal surface with respect to $T$.
 Among all the non-separating surfaces with maximal Euler characteristic
which are normal with respect to $T$,
let $F'$ be one with $w(F')=|F' \cap T^{(1)}|$ minimal,
where $|F' \cap T^{(1)}|$ is the number of intersection points 
of $F'$ and the $1$-skelton of $T$.

 We will show that $F'$ is fundamental.
 Suppose, for a contradiction, that $F'$ is not.
 Then there are normal surfaces $F_1, F_2$ with $F' = F_1 + F_2$.
 We can assume, without loss of generality,
that this Haken sum is in reduced form.
 By Lemma 2.1 in \cite{JO}, 
the Haken sum $F_1 + F_2$ has no disk patches.
 Then both $F_1$ and $F_2$ has Euler characteristics 
larger than or equal to that of $F'$.
 Moreover, 
either $F_1$ or $F_2$, say $F_1$ is non-separating by Lemma \ref{lemma:switch}.
 Since $w(F') = w(F_1) + w(F_2) > w(F_1)$,
we obtain a contradiction to the minimality of $w(F)$.
\end{proof}

\begin{proposition}\label{proposition:Q-Fund}
 Let $M$ be a closed $3$-manifold with a triangulation $T$.
 If $M$ contains a non-separating closed surface $F$,
then it contains a non-separating closed surface $F'$ 
such that $F'$ is Q-fundamental with respect to $T$.
\end{proposition}

\begin{proof}
 We assume that the readers have good familiality with the paper \cite{T}.

 As in the proof of Proposition \ref{proposition:Fund},
$M$ contains a normal non-separating closed surface.
 Let $F'$ be one with minimal number of Q-disks 
among all the non-separating surfaces which are normal.
 We will show that $F'$ is Q-fundamental.

 Assume, for a contradiction, that $F'$ is not Q-fundamental.
 Then there are non-trivial normal surfaces $F_1, F_2$ 
and a union of trivial normal surface $\Sigma$
such that $F' + \Sigma = F_1 + F_2$.
 (Recall that a normal surface is called trivial 
if it consists of trigonal normal disks and has no quadrilateral normal disks,
and that a normal surface is determined by a Q-coordinate
up to trivial components.)
 Since $F'$ is non-separating,
$F' + \Sigma = F' \sqcup \Sigma$ is also non-separating in $M$.
 If $F_1$ and $F_2$ are both separating,
then Lemma \ref{lemma:switch} shows that $F' + \Sigma$ is separating,
which is a contradiction.
 Hence $F_1$ or $F_2$, say $F_1$ is non-separating,
and has a non-separating component $F^*$.
 However, $F^*$ has smaller or equal number of Q-disks than $F_1$, 
and $F_1$ has strictly smaller number of Q-disks than $F_1 + F_2 = F + \Sigma$.
 This is a contradiction.
\end{proof}


\section{The (2,1)-lens space}

 In this section we prove Theorem \ref{theorem:main} (1).
 For the triangulation $T(2,1)$ of the $(2,1)$-lens space,
the senses of the types the quadrilateral disks $X_{k1}, X_{k2}, X_{k3}$ 
in $\tau_i$
are as below, 
where $\{ j,k \} = \{ 1,2 \}$
and $\epsilon_{k,ki}$, $\epsilon_{j,ki}$, 
$\epsilon_{E_h,ki}$, $\epsilon_{E_V,ki}$
denote the senses of $X_{ki}$
with respect to the edges $e_k$, $e_j$, $E_h$, $E_v$ respectively.
\newline
$\begin{array}{lll}
\epsilon_{k,k1}=-2, & \epsilon_{k,k2}=2, & \epsilon_{k,k3}=0, \\
\epsilon_{j,k1}=2, & \epsilon_{j,k2}=0, & \epsilon_{j,k3}=-2, \\
\epsilon_{E_h,k1}=0, & \epsilon_{E_h,k2}=-1, & \epsilon_{E_h,k3}=1, \\
\epsilon_{E_v,k1}=0, & \epsilon_{E_v,k2}=-1, & \epsilon_{E_v,k3}=1 \\
\end{array}$
\newline
 Then the coefficient matrix of the Q-matching equations is as below.
\newline
$\left( \begin{array}{ccc|ccc}
\epsilon_{1,11}   & \epsilon_{1,12}   & \epsilon_{1,13}   &
\epsilon_{1,21}   & \epsilon_{1,22}   & \epsilon_{1,23}   \\
\epsilon_{2,11}   & \epsilon_{2,12}   & \epsilon_{2,13}   &
\epsilon_{2,21}   & \epsilon_{2,22}   & \epsilon_{2,23}   \\
\epsilon_{E_h,11} & \epsilon_{E_h,12} & \epsilon_{E_h,13} &
\epsilon_{E_h,21} & \epsilon_{E_h,22} & \epsilon_{E_h,23} \\
\epsilon_{E_v,11} & \epsilon_{E_v,12} & \epsilon_{E_v,13} &
\epsilon_{E_v,21} & \epsilon_{E_v,22} & \epsilon_{E_v,23} \\
\end{array} \right)
=
\left( \begin{array}{ccc|ccc}
-2 & 2 & 0 & 2 & 0 &-2 \\
 2 & 0 &-2 &-2 & 2 & 0 \\
 0 &-1 & 1 & 0 &-1 & 1 \\
 0 &-1 & 1 & 0 &-1 & 1 \\
\end{array} \right)$
\newline 
 Adding half of the sum of the first row and the second row
to the third row and to the fourth row,
we obtain the matrix
\newline
$\left( \begin{array}{ccc|ccc}
-2 & 2 & 0 & 2 & 0 &-2 \\
 2 & 0 &-2 &-2 & 2 & 0 \\
 0 & 0 & 0 & 0 & 0 & 0 \\
 0 & 0 & 0 & 0 & 0 & 0 \\
\end{array} \right)$.
\newline 
 The dimension of the solution space is $6-2=4$
since the rank of the matrix is $2$.
 The four vectors below together form a basis of the solution space
in ${\Bbb R}^6$.
\newline
${\bf v}_a= {}^t(1,1,1\ |\ 0,0,0)$, 
${\bf v}_b= {}^t(0,0,0\ |\ 1,1,1)$, 
${\bf v}_c= {}^t(0,1,0\ |\ 0,0,1)$, 
${\bf v}_d= {}^t(0,0,1\ |\ 0,1,0)$ 

 In the rest, we determine all the Q-fundamental surfaces 
for the triangulation $T(2,1)$ of the $(2,1)$-lens space.
 Set ${\bf v} = a{\bf v}_a + b{\bf v}_b + c{\bf v}_c + d{\bf v}_d =
{}^t (a,\ a+c,\ a+d \ |\ b,\ b+d,\ b+c)$
with $a,b,c,d \in {\Bbb R}$,
the general solution of the matching equations.
 We consider when ${\bf v}$ represents a Q-fundamental surface.

\vspace{2mm}
\noindent
{\bf Case of $a>0$}: 
 Since the first entry is positive,
the second and third entries are $0$ by the square condition.
 Then we get $c=d=-a$,
and fourth and the fifth entries are $b-a$.
 The square condition implies these entires are $0$
(otherwise, both of them would be positive),
and we have $b=a$ and 
${\bf v}={}^t (a,0,0\ |\ a,0,0)$. 
$a \ge 1$ implies 
${\bf v}={}^t (a,0,0\ |\ a,0,0) \ge {}^t (1,0,0\ |\ 1,0,0)$.
 If ${\bf v}$ is Q-fundamental, 
then ${\bf v}={}^t (1,0,0\ |\ 1,0,0)$.
 Let ${\bf f}_1$ denote this vector, 
which is a candidate of a Q-fundamental solution.

\vspace{2mm}
\noindent
{\bf Case of $a=0$}: 
 In this case, the second, the third and the fourth entries 
are $d,c$ and $b$ respectively. 
 Hence $b,c,d$ are non-negative integers.
 Thus, if ${\bf v} = b{\bf v}_b + c{\bf v}_c + d{\bf v}_d$
represents a Q-fundamental solution,
then $(b,c,d)=(1,0,0), (0,1,0)$ or $(0,0,1)$,
and ${\bf v}={\bf v}_b, {\bf v}_c$ or ${\bf v}_d$.
 ${\bf v}_b$ violates the square condition.
 We set ${\bf f}_2={\bf v}_c=(0,1,0\ |\ 0,0,1)$, 
${\bf f}_3={\bf v}_d = (0,0,1\ |\ 0,1,0)$
as in Theorem \ref{theorem:main} (i). 
 Note that $\pi$-rotation about the axis $E_v$ carries 
${\bf f}_3$ to ${\bf f}_2$.

 Actually, the three vectors ${\bf f}_1, {\bf f}_2, {\bf f}_3$ 
are all Q-fundamental solutions.
 To see this,
for any vector ${\bf x}$,
let size$({\bf x})$ be the sum of all the elements of ${\bf x}$.
 Since size$({\bf f}_1)=2$, size$({\bf f}_2)=2$ and size$({\bf f}_3)=2$, 
each of the three vectors can't be presented
as a non-trivial linear combination of other candidates
with non-negative integer coefficients.

 ${\bf f}_1$ represents a Heegaard splitting torus 
which surrounds the core circle $E_v$. 
 ${\bf f}_2$ and ${\bf f}_3$ represent projective planes.
 $2{\bf f}_2$ and $2{\bf f}_3$ represent the inessential $2$-spheres
surrounding the edges $e_2$ and $e_1$ respectively.


\section{$(p,1)$-lens spaces}

 In this section,
we prove Theorem \ref{theorem:main} (2),
that is,
we determine all the Q-fundamental surfaces
for the triangulation $T(p,1)$ of the $(p,1)$-lens space with $p \ge 3$.

\begin{figure}[htbp]
\begin{center}
\includegraphics[width=10cm]{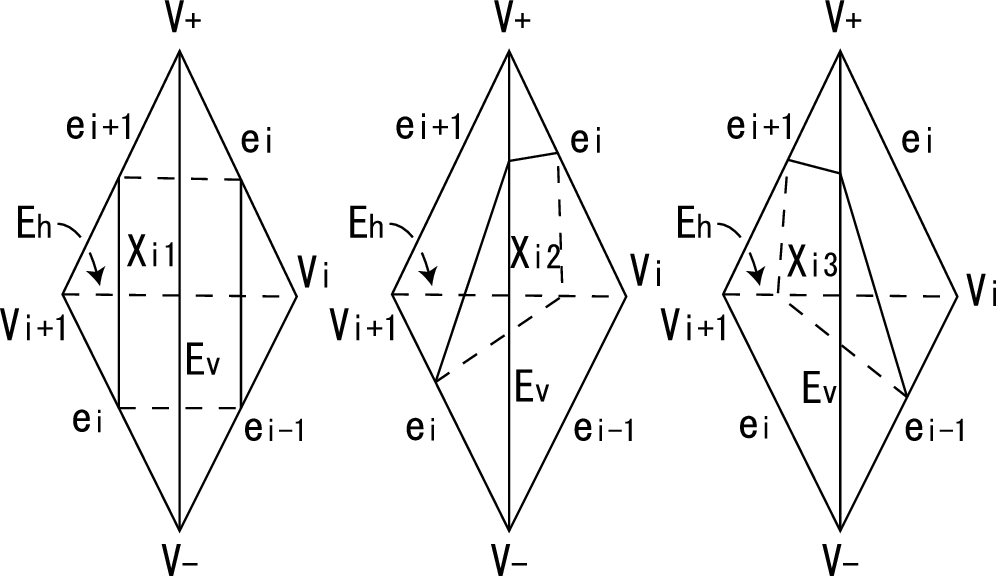}
\end{center}
\caption{}
\label{fig:numberQdiskp1}
\end{figure}

 For $T(p,1)$, 
the senses of $X_{i1}, X_{i2}, X_{i3}$
with respect to an edge $e$
are as below.
\newline
$\epsilon_{e,i1}=
\left\{ \begin{array}{l}
1\ {\rm if}\ e=e_{i-1} \\
-2\ {\rm if}\ e=e_i \\
1\ {\rm if}\ e=e_{i+1} \\
0\ {\rm if}\ e=E_h \\
0\ {\rm if}\ e=E_v \\
0\ {\rm otherwise} \\
\end{array} \right.\ $
$\epsilon_{e,i2}=
\left\{ \begin{array}{l}
0\ {\rm if}\ e=e_{i-1} \\
2\ {\rm if}\ e=e_i \\
0\ {\rm if}\ e=e_{i+1} \\
-1\ {\rm if}\ e=E_h \\
-1\ {\rm if}\ e=E_v \\
0\ {\rm otherwise} \\
\end{array} \right.\ $
$\epsilon_{e,i3}=
\left\{ \begin{array}{l}
-1\ {\rm if}\ e=e_{i-1} \\
0\ {\rm if}\ e=e_i \\
-1\ {\rm if}\ e=e_{i+1} \\
1\ {\rm if}\ e=E_h \\
1\ {\rm if}\ e=E_v \\
0\ {\rm otherwise} \\
\end{array} \right.\ $
\newline
where suffix numbers are considered modulo $p$.
 Note that $\epsilon_{i,i1}=-2$ 
since the tetrahedron $\tau_i$ has two copies of the edge $e_i$,
and the disk of type $i1$ intersects each of the copies 
at a single point with negative sign.
 Then we obtain the coefficient matrix of the Q-matching equations 
for the $(p,1)$-lens space as below.
\newline
$\left( \begin{array}{ccc|ccc|ccc|ccc|ccc|ccc}
-2 & 2 & 0 & 1 & 0 &-1 & 0 & 0 & 0 &   &\cdots &   & 0 & 0 & 0 & 1 & 0 &-1 \\
 1 & 0 &-1 &-2 & 2 & 0 & 1 & 0 &-1 & 0 & 0 & 0 &   &\cdots &   & 0 & 0 & 0 \\
 0 & 0 & 0 & 1 & 0 &-1 &-2 & 2 & 0 & 1 & 0 &-1 &   &   &   &   &   &   \\
   &   &   &   &   &   &   &\ddots &   &   & \ddots&   &   & \ddots&   &   \\
   &   &   &   &   &   &   &   &   & 1 & 0 &-1 &-2 & 2 & 0 & 1 & 0 &-1 \\
 1 & 0 &-1 &   &   &   &   &   &   &   &   &   & 1 & 0 &-1 &-2 & 2 & 0 \\
 0 &-1 & 1 & 0 &-1 & 1 & 0 &-1 & 1 &   &\cdots &   & 0 &-1 & 1 & 0 &-1 & 1 \\
 0 &-1 & 1 & 0 &-1 & 1 & 0 &-1 & 1 &   &\cdots &   & 0 &-1 & 1 & 0 &-1 & 1 \\
\end{array} \right)$
 Let $r_i$ be the $i$-th row.
 Adding $(r_1 + r_2 + \cdots + r_p)/2$
to the $(p+1)$-st row and to the $(p+2)$-nd row,
we can reduce the matrix as below.
\newline
$\left( \begin{array}{ccc|ccc|ccc|ccc|ccc|ccc}
-2 & 2 & 0 & 1 & 0 &-1 & 0 & 0 & 0 &   &\cdots &   & 0 & 0 & 0 & 1 & 0 &-1 \\
 1 & 0 &-1 &-2 & 2 & 0 & 1 & 0 &-1 & 0 & 0 & 0 &   &\cdots &   & 0 & 0 & 0 \\
 0 & 0 & 0 & 1 & 0 &-1 &-2 & 2 & 0 & 1 & 0 &-1 &   &   &   &   &   &   \\
   &   &   &   &   &   &   &\ddots &   &   & \ddots&   &   & \ddots&   &   \\
   &   &   &   &   &   &   &   &   & 1 & 0 &-1 &-2 & 2 & 0 & 1 & 0 &-1 \\
 1 & 0 &-1 &   &   &   &   &   &   &   &   &   & 1 & 0 &-1 &-2 & 2 & 0 \\
 0 & 0 & 0 & 0 & 0 & 0 & 0 & 0 & 0 &   &\cdots &   & 0 & 0 & 0 & 0 & 0 & 0 \\
 0 & 0 & 0 & 0 & 0 & 0 & 0 & 0 & 0 &   &\cdots &   & 0 & 0 & 0 & 0 & 0 & 0 \\
\end{array} \right)$
\newline
 The rank of this matrix is $p$ 
since the $2$nd element of the $i$-th block 
(the $(3(i-1)+2)$-nd element) 
is non-zero 
only for the $i$-th row vector
for $1 \le i \le p$. 
 Hence the dimension of the solution space is $3p-p=2p$.
 We obtain a basis ${\bf s}'_1, {\bf s}'_2, \cdots, {\bf s}'_p$,
${\bf t}'_1, {\bf t}'_2, \cdots, {\bf t}'_p$ of the solution space
considered in ${\Bbb R}^{3p}$, where
\newline
${\bf s}'_1 = {}^t (1,1,1\ |\ 0,0,0\ | 0,0,0\ |\ \cdots |\ 0,0,0\ |\ 0,0,0)$,
\newline
${\bf s}'_2 = {}^t (0,0,0\ |\ 1,1,1\ | 0,0,0\ |\ \cdots |\ 0,0,0\ |\ 0,0,0)$,
\newline
$\vdots$
\newline
${\bf s}'_p = {}^t (0,0,0\ |\ 0,0,0\ | 0,0,0\ |\ \cdots |\ 0,0,0\ |\ 1,1,1)$,
\newline
${\bf t}'_1 
={}^t(0,0,2\ |\ 0,1,0\ | 0,0,0\ |\ 0,0,0\ |\ \cdots 0,0,0\ |\ 0,0,0\ |\ 0,1,0)$,
\newline
${\bf t}'_2 
={}^t(0,1,0\ |\ 0,0,2\ | 0,1,0\ |\ 0,0,0\ |\ \cdots 0,0,0\ |\ 0,0,0\ |\ 0,0,0)$,
\newline
$\vdots$
\newline
${\bf t}'_p 
={}^t(0,1,0\ |\ 0,0,0\ | 0,0,0\ |\ 0,0,0\ |\ \cdots 0,0,0\ |\ 0,1,0\ |\ 0,0,2)$.
\newline
 Note that ${\bf s}'_i$ and ${\bf t}'_i$ are obtained 
from ${\bf s}'_1$ and ${\bf t}'_1$ respectively
by ($2\pi(i-1)/p$)-rotation about the axis $E_v$. 
 Each ${\bf s}'_i$ is a solution of the Q-matching equations
for any triangulation of any $3$-manifold.
 (The sum of 
the three columns of the $k$-th block
(the $(3(k-1)+1)$-st, the $(3(k-1)+2)$-nd and the $(3(k-1)+3)$-rd columns)
is zero
for $1 \le k \le p$.)

\begin{figure}[htbp]
\begin{center}
\includegraphics[width=10cm]{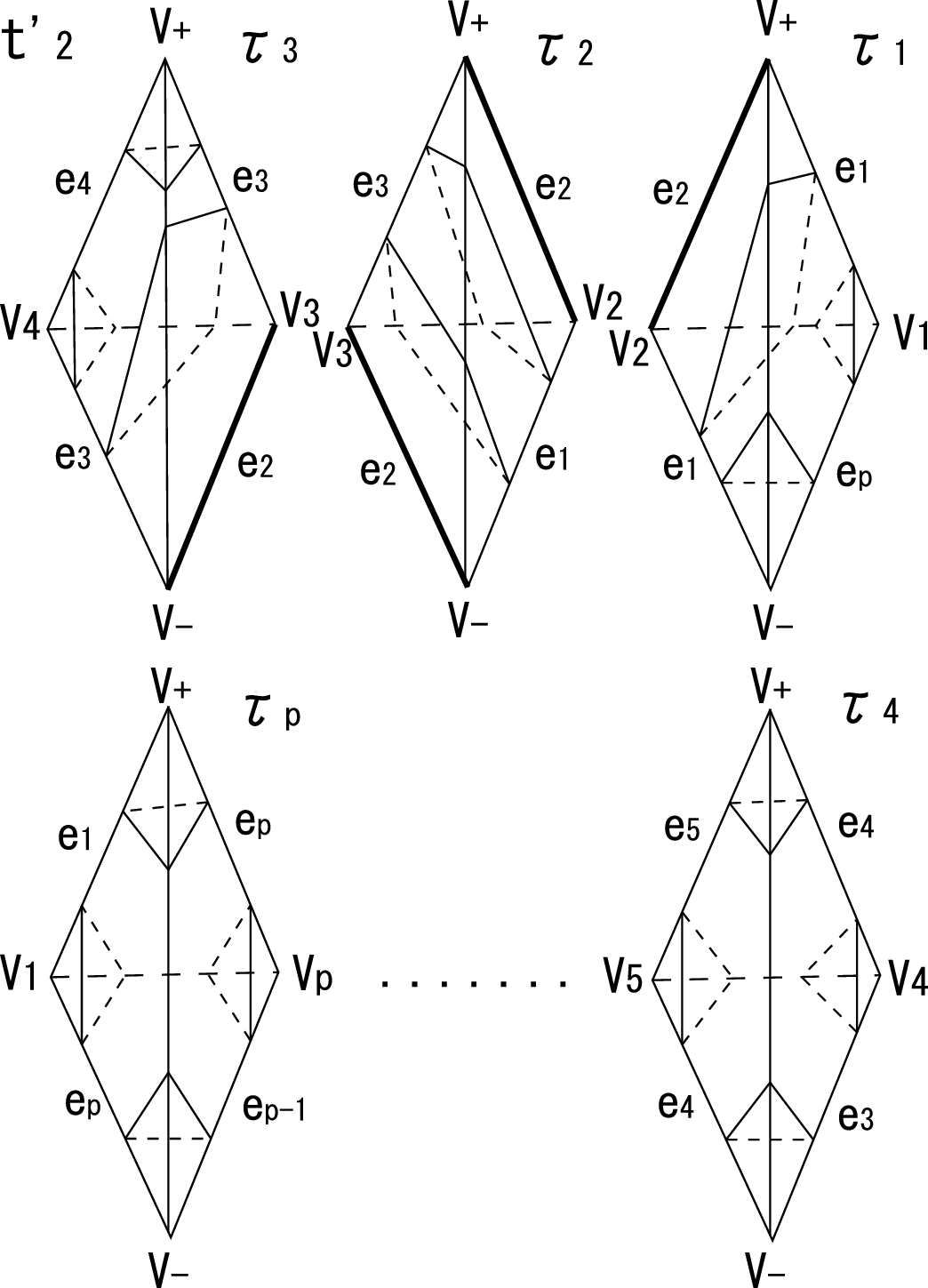}
\end{center}
\caption{}
\label{fig:inessS2forP1}
\end{figure}

 ${\bf t}'_i$ represents the normal surface
which is the inessential $2$-sphere surrounding the edge $e_i$
(see Figure \ref{fig:inessS2forP1}),
and hence is a solution of the Q-matching equations.
 We shall see these vectors are linearly independent.
 We set the general solution 
\newline
${\bf v} = a_1 {\bf s}'_1 + a_2 {\bf s}'_2 + \cdots + a_p {\bf s}'_p
+ b_1 {\bf t}'_1 + b_2 {\bf t}'_2 + \cdots + b_p {\bf t}'_p=$
\newline
$= 
{}^t (a_1,\ a_1+b_2+b_p,\ a_1+2b_1 \ | \cdots |\ 
a_i,\ a_i+b_{i+1}+b_{i-1},\ a_i+2b_i \ | \cdots $
\newline
\hspace*{2cm}
$\cdots |\ a_p,\ a_p+b_1+b_{p-1},\ a_p+2b_p)$
\newline
with $a_1, \cdots, a_p, b_1, \cdots, b_p\in {\Bbb R}$.
 If ${\bf v}=0$, 
considering the $1$st element of the $i$-th block,
we have $a_i=0$ for all i.
 Then $b_k=0$ follows from the $3k$ entry ($1\le k \le p$).
 Hence the $2p$ vectors are linearly independent. 

 We consider 
when the general solution ${\bf v}$ represents a Q-fundamental surface.
 Since ${\bf v}$ represents a normal surface,
all the elements of ${\bf v}$ are non-negative integers.
 From the $1$st element of the $i$-th block, 
$a_i$ is a non-negative integer for $1 \le i \le p$.
 Then $b_i \in {\Bbb Z}/2$
because the $3$rd element of the $i$-th block $a_i + 2b_i$ is an integer,
where ${\Bbb Z}/2$ is the set of all the integers and the half integers.

\vspace{2mm}
\noindent
{\bf The case of $b_i \ge 0$ for all $i$}: 
 We consider the case where $b_i \ge 0$ for $1 \le i \le p$.
 If $a_i$ were positive, 
then both the $1$st element of the $i$-th block $a_i$ 
and the $3$rd element of the $i$-th block $a_i + 2b_i$ would be positive, 
contradicting the square condition.
 Hence $a_i = 0$ for all $i$. 
 Then 
\newline
${\bf v}= 
b_1{\bf t}'_1 + b_2{\bf t}'_2 + \cdots + b_p{\bf t}'_p$
\newline
\hspace*{3mm}
$=
{}^t (0,\ b_2+b_p,\ 2b_1\ |\ 0,\ b_3+b_1,\ 2b_2\ | \cdots 
|\ 0,\ b_{i+1}+b_{i-1},\ 2b_i\ | \cdots |\ 0,\ b_1+b_{p-1},\ 2b_p ) \cdots (*)$
\newline
(i) We consider the case where $b_i\ge 1$ for some $i$. 
\newline
${\bf v}\ge 
{\bf t}'_i = 
{}^t (0,0,0\ | \cdots |\ 0,0,0\ |\ 0,1,0\ |\ 0,0,2\ |\ 0,1,0\ |\ 0,0,0\ | 
\cdots |\ 0,0,0)$ by $(*)$,
where the $i$-th block is $(0,0,2)$.
 When ${\bf v}$ represents a Q-fundamental surface,
${\bf v}={\bf t}'_i$.
 Thus we obtained a candidate of a Q-fundamental solution.
\newline
(ii) The case where $b_i<1$ for all i.

 We have $b_k=0$ or $1/2$ for each $k$ since $b_k \in {\Bbb Z}/2$.

 If $b_j=1/2$, then $b_{j+1}=0$. 
 Otherwise, 
$b_j=1/2$ and $b_{j+1}=1/2$ for some j, 
and both the $2$nd element of the $j$-th block $b_{j+1}+b_{j-1}$ 
and the $3$rd element of the $j$-th block $2b_j$ would be positive, 
contradicting the square condition. 

 If $b_k=1/2$ for some $k$, then $b_{k+2}=1/2$. 
 Because the $2$nd element $b_{k+2}+b_k$ of the $(k+1)$-st block is an integer. 

 Hence $p$ is even 
and
$b_{\rm odd} = 1/2$ and $b_{\rm even}=0$
or vice versa ($b_{\rm odd} = 0$ and $b_{\rm even}=1/2$).
 Thus we obtained candidates of Q-fundamental solutions 
\newline
${\bf f}'_2=(1/2){\bf t}'_1 + (1/2){\bf t}'_3 + \cdots + (1/2){\bf t}'_{p-1}=
{}^t(0,0,1\ |\ 0,1,0\ |\ 0,0,1\ |\ 0,1,0\ | \cdots |\ 0,0,1\ |\ 0,1,0)$ and
\newline
${\bf f}'_3=(1/2){\bf t}'_2 + (1/2){\bf t}'_4 + \cdots + (1/2){\bf t}'_p=
{}^t(0,1,0\ |\ 0,0,1\ |\ 0,1,0\ |\ 0,0,1\ | \cdots |\ 0,1,0\ |\ 0,0,1)$.

\vspace{3mm}
\noindent
{\bf The case where $b_i< 0$ for some $i$}: 

 First, we establish the next lemma.

\begin{lemma}\label{lemma:p1} 
If $b_i<0$ for some $i$, 
then $a_i+2b_i=0$, $a_i+b_{i+1}+b_{i-1}=0$
and $b_1=b_2=b_3=\cdots=b_p$. 
\end{lemma}

\begin{proof}
 Since $b_i$ is negative,
and since the $3$rd element of the $i$-th block $a_i+2b_i$ is non-negative,
$a_i\ge -2b_i>0$. 
 Then the $1$st element of the $i$-th block $a_i$ is positive,
and the square condition requires the other two elements of the block are zero,
i.e., $a_i+2b_i=0$ and $a_i+b_{i+1}+b_{i-1}=0$. 
 By subtracting the former equation from the latter,
we obtain $b_{i-1} - b_i = b_i - b_{i+1}$.
 Hence either $b_{i-1}<b_i<b_{i+1}$, 
$b_{i-1}=b_i=b_{i+1}$ 
or $b_{i-1}>b_i>b_{i+1}$. 
 Then $b_{i-1}\ge b_i \ge b_{i+1}$ or $b_{i-1}\le b_i \le b_{i+1}$ holds. 
 We consider the case of $b_{i-1}\ge b_i \ge b_{i+1}$. 
 (Similar argument will do for the case of $b_{i-1}\le b_i \le b_{i+1}$. 
 We omit it.) 
 We will prove $0>b_i\ge b_{i+1}\ge \cdots \ge b_{i+n}$ 
for any positive integer $n$ by induction. 
 Suppose that $0>b_i\ge b_{i+1}\ge \cdots \ge b_{i+k} \ \ \cdots$ (i). 
 Since $b_{i+k}$ is negative 
we have $b_{i+k-1}<b_{i+k}<b_{i+k+1}$, 
$b_{i+k-1}=b_{i+k}=b_{i+k+1}$ 
or $b_{i+k-1}>b_{i+k}>b_{i+k+1}$
by a similar argument as the beginning of this proof. 
 Then (i) implies $b_{i+k-1} \ge b_{i+k} \ge b_{i+k+1}$. 
 Hence $0>b_i\ge b_{i+1}\ge \cdots \ge b_{i+n}$ for all $n \in {\Bbb N}$. 
 Considering the case of $n=p$,
where $b_i \ge b_{i+1}\ge \cdots \ge b_{i+p}=b_i$,
we obtain $b_i = b_{i+1} = \cdots = b_{i+p-1}$.
\end{proof}

 By the above lemma,
$b_1=b_2=\cdots=b_p$ $\cdots$ (i), 
and hence $b_j<0$ for all $j$.
 Then $a_j + 2b_j=0$ follows again by Lemma \ref{lemma:p1}.
 Hence (i) implies $a_j=-2b_j=-2b_1$ for all $j$,
and $a_1 = a_2 = \cdots = a_p$. 
 Thus 
\newline
${\bf v}
=   {}^t (a_1,0,0\ |\ a_1,0,0\ | \cdots |\ a_1,0,0)
\ge {}^t (1,0,0\ |\ 1,0,0\ | \cdots |\ 1,0,0) = {\bf f}'_1$.
\newline
Since ${\bf v}$ is Q-fundamental, ${\bf v} = {\bf f}'_1$, 
which is a candidate of a Q-fundamental solution.

\vspace{3mm}
 Thus we have obtained all the candidates of Q-fundamental solutions
${\bf t}'_i, {\bf f}'_2, {\bf f}'_3, {\bf f}'_1$ 
satisfying the square condition. 
 In fact, these four vectors are all Q-fundamental solutions.


 Since size$({\bf t}'_i)=4$, size$({\bf f}'_2)=p$, 
size$({\bf f}'_3)=p$ and size$({\bf f}'_1)=p$, 
and since $p\ge 3$,
the minimal size of a non-trivial linear combination of these vectors
with non-negative integer coefficients is $3+3=6$.
 Hence ${\bf t}'_i$ is Q-fundamental.

 ${\bf f}'_1$ is Q-fundamental
because the $1$st element is non-zero
and those of the other candidates are zero.

 ${\bf f}'_2$ is Q-fundamental
since the $3$rd element of the $2$nd block is $1$
and those of the other candidates are $0$ or $2$.

 ${\bf f}'_3$ is Q-fundamental
because $(2\pi (i-1)/p)$-rotation about the axis $E_v$ 
brings ${\bf f}'_3$ to ${\bf f}'_2$.

 ${\bf t}'_i$ represents the inessential $2$-sphere surrounding the edge $e_i$
(Figure \ref{fig:inessS2forP1}). 
 ${\bf f}'_2$ and ${\bf f}'_3$ represent a non-orientable closed surface 
with maximum Euler characteristic, 
which is the connected sum of $p/2$ projective planes ($p$ is even)
(Figure \ref{fig:Lp1non-oriXmax}).
 ${\bf f}'_1$ represents a Heegaard splitting torus 
which surrounding the core circle $E_v$
(Figure \ref{fig:Lp1heegaardtorus}). 

\begin{figure}[htbp]
\begin{center}
\includegraphics[width=9cm]{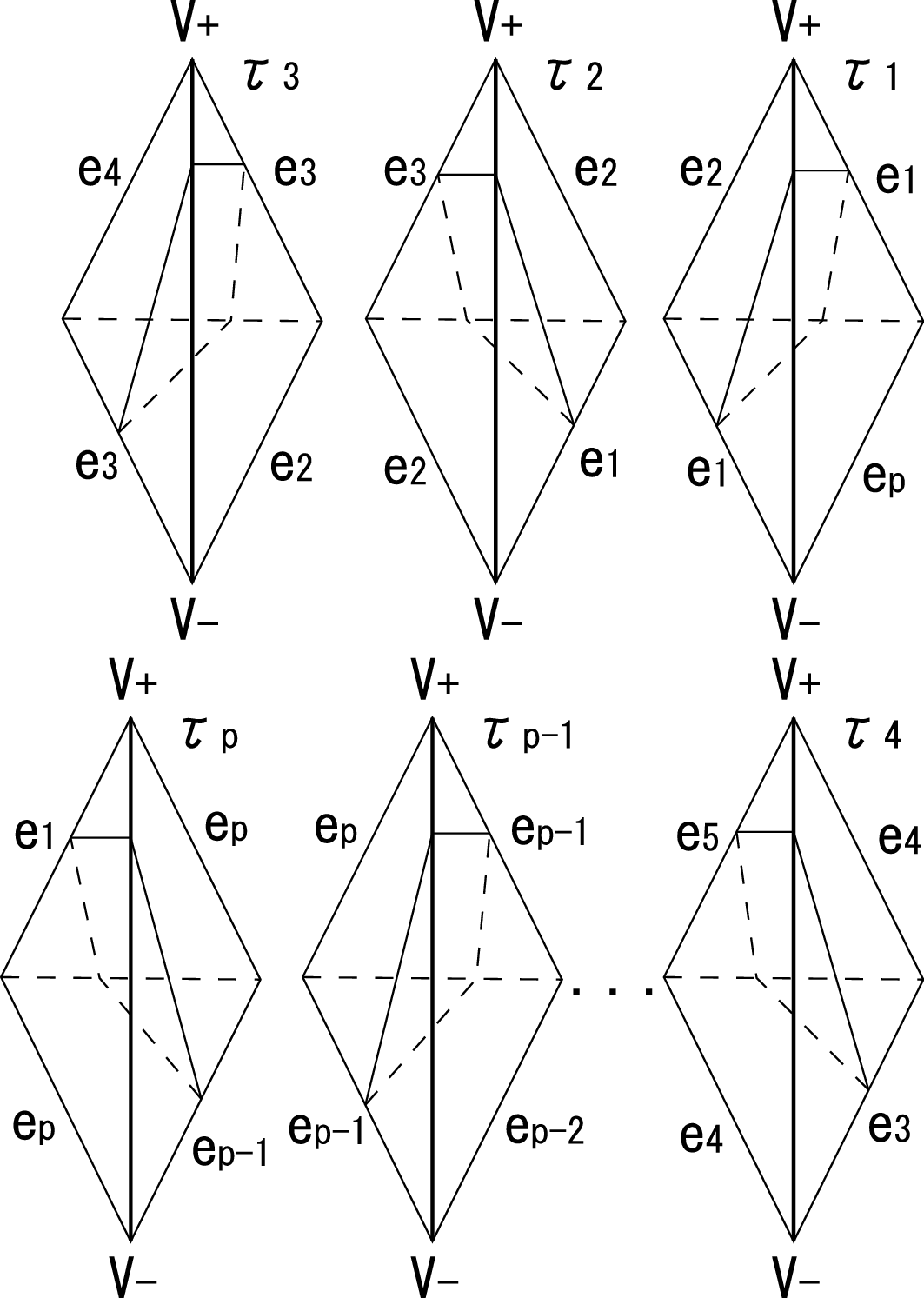}
\end{center}
\caption{}
\label{fig:Lp1non-oriXmax}
\end{figure}

\begin{figure}[htbp]
\begin{center}
\includegraphics[width=6cm]{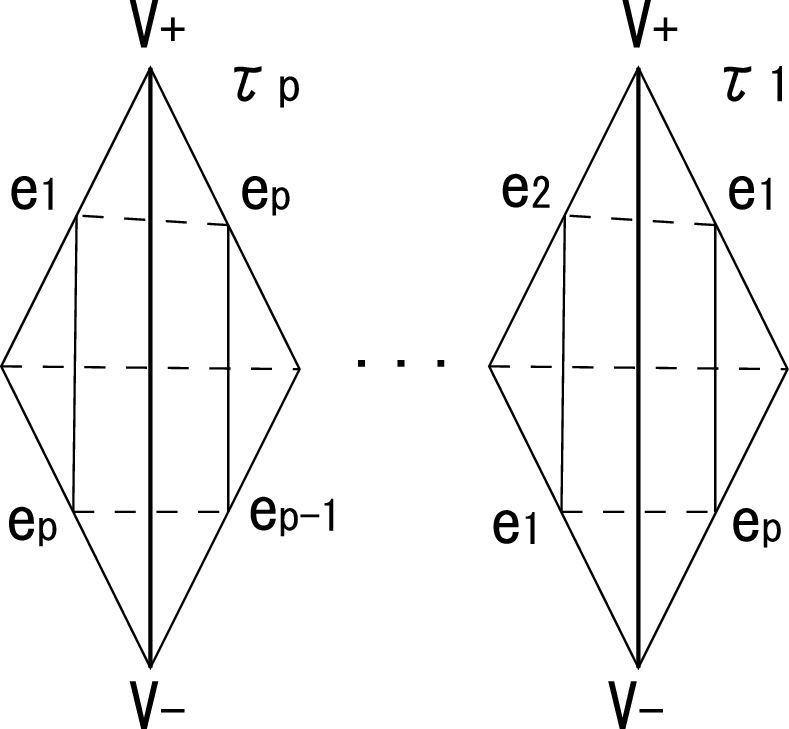}
\end{center}
\caption{}
\label{fig:Lp1heegaardtorus}
\end{figure}


\section{A basis of the solution space of the Q-matching equations 
for $T(p,q)$}

 In this section, 
we prove Lemma \ref{lemma:generator}.


 For $T(p,q)$, 
the senses of $X_{i1}, X_{i2}, X_{i3}$
with respect to an edge $e$
are as below.
 See Figure \ref{fig:numberQdisk}.
\newline
$\epsilon_{e,i1}=
\left\{ \begin{array}{l}
1\ {\rm if}\ e=e_{i-q} \\
-1\ {\rm if}\ e=e_{i-(q-1)} \\
-1\ {\rm if}\ e=e_i \\
1\ {\rm if}\ e=e_{i+1} \\
0\ {\rm if}\ e=E_h \\
0\ {\rm if}\ e=E_v \\
0\ {\rm otherwise} \\
\end{array} \right.\ $
$\epsilon_{e,i2}=
\left\{ \begin{array}{l}
0\ {\rm if}\ e=e_{i-q} \\
1\ {\rm if}\ e=e_{i-(q-1)} \\
1\ {\rm if}\ e=e_i \\
0\ {\rm if}\ e=e_{i+1} \\
-1\ {\rm if}\ e=E_h \\
-1\ {\rm if}\ e=E_v \\
0\ {\rm otherwise} \\
\end{array} \right.\ $
$\epsilon_{e,i3}=
\left\{ \begin{array}{l}
-1\ {\rm if}\ e=e_{i-q} \\
0\ {\rm if}\ e=e_{i-(q-1)} \\
0\ {\rm if}\ e=e_i \\
-1\ {\rm if}\ e=e_{i+1} \\
1\ {\rm if}\ e=E_h \\
1\ {\rm if}\ e=E_v \\
0\ {\rm otherwise} \\
\end{array} \right.\ $
\newline
where suffix numbers are considered modulo $p$.
 The coefficient matrix of the Q-matching equations 
is a $(p+2)\times 3p$ matrix. 
 For the $i$-th row $(1\le i \le p)$, 
the $i$-th and the $(q+(i-1))$-st blocks are $(-1,1,0)$, 
the $(i-1)$-st and the $(q+i)$-th blocks are $(1,0,-1)$ and 
the other blocks are $(0,0,0)$. 
 All the blocks of the $(p+1)$-st and the $(p+2)$-nd rows are $(0,-1,1)$. 

 Let $r_i$ be the $i$-th row.
 By adding $(r_1 + r_2 + \cdots + r_p)/2$
to the $(p+1)$-st row and to the $(p+2)$-nd row, 
we can deform these rows to zero vectors.
 Hence the rank is smaller than or equal to $p$. 
 We set 
\newline
${\bf s}_1=
{}^t (1,1,1\ |\ 0,0,0\ | \cdots |\ 0,0,0)$, 
\newline
${\bf t}_1=
{}^t (0,0,1\ |\ 0,0,0\ | \cdots |\ 0,0,0\ |\ 0,0,1\ |\ 0,1,0\ |\ 0,0,0\ |
\cdots |\ 0,1,0)$.
\newline
 The $1$st and the $q$-th blocks of ${\bf t}_1$ are $(0,0,1)$
and the $(q+1)$-st and the $p$-th blocks $(0,1,0)$.
 The other blocks of ${\bf t}_1$ are $(0,0,0)$.
 As we see below, these are solutions of the Q-matching equations.
 ${\bf s}_i$ and ${\bf t}_i$ are obtained 
from ${\bf s}_1$ and ${\bf t}_1$ respectively
by ($2\pi(i-1)/p$)-rotation about the axis $E_v$. 
 The $i$-th and the $(i+(q-1))$-st blocks of ${\bf t}_i$ are $(0,0,1)$
and the $(i-1)$-st and the $(i+q)$-th blocks $(0,1,0)$.
 ${\bf s}_i$ is a solution of the Q-matching equations
for any triangulation of any $3$-manifold. 
 However, it does not satisfy the square condition.
 ${\bf t}_i$ represents the normal surface 
which is the inessential $2$-sphere surrounding the edge $e_i$ 
and hence is a solution of the Q-matching equations. 

 We prove that ${\bf s}_1, {\bf s}_2, \cdots, {\bf s}_p$,
${\bf t}_1, {\bf t}_2, \cdots, {\bf t}_p$ are linearly independent. 
 We solve the system of linear equations below. 
\newline
${\bf 0} = 
a_1 {\bf s}_1 
+ a_2 {\bf s}_2 
+ \cdots
+ a_p{\bf s}_p 
+ b_1 {\bf t}_1 
+ b_2 {\bf t}_2 
+ \cdots 
+ b_p {\bf t}_p$
\newline
\hspace*{3mm}
$= 
{}^t (
a_1,\ a_1+b_2+b_{p-q+1},\ a_1+b_1+b_{p-q+2}\ |\ 
a_2,\ a_2+b_3+b_{p-q+2},\ a_2+b_2+b_{p-q+3}\ |
\cdots $
\newline
$\ \ \cdots |\ 
a_i,\ a_i+b_{i+1}+b_{p-q+i},\ a_i+b_i+b_{p-q+i+1}\ |
\cdots |\ 
a_p,\ a_p+b_1+b_{p-q},\ a_p+b_p+b_{p-q+1}\ )$.
\newline
 We have $a_i=0$ from the $1$st element of the $i$-th block $(1 \le i \le p)$. 
 Hence from the $2$nd and the $3$rd elements of the $i$-th block
$b_{i+1}+b_{p-q+i}=0 \ \cdots $(1) and $b_i+b_{p-q+i+1}=0 \ \cdots $(2). 
 Since this holds for all $i$ and since we consider suffix numbers modulo $p$, 
we obtain the equation (3) below. 
(Substituting $j+q$ for $i$ in (1), we have $b_{p+j}=-b_{j+q+1}$.
 Then the first equation of (3) follows. 
 For the second equation of (3), 
we replace $i$ with $j+q+1$ in (2). 
 Then $-b_{j+q+1}=b_{p+j+(q+1)-(q-1)}$ holds.)
\begin{center}
$b_j=-b_{j+(q+1)}=b_{j+(q+1)-(q-1)}=b_{j+2}\ \cdots {\rm (3)}$
\end{center}

\noindent
(i) The case of $q=2k+1$ $(k \in {\Bbb N})$. 
\par
 We get $b_i=b_{i+2}=b_{i+4}=\cdots =b_{i+2k}=b_{i+q-1}\ \cdots $(4).
 From (2), $b_i=-b_{i+q-1}$ follows. 
 This together with (4) implies $b_i=-b_i$, that is, $b_i=0$ for all $i$. 
 Hence the vectors are linearly independent.
\newline
(ii) The case of $q=2m$ $(m \in {\Bbb N})$. 
\par
$b_i=-b_{i+(q+1)}$ follows from (1). 
 On the other hand, 
$-b_{i+(q+1)}=-b_{i+2m+1}=-b_{i+2(m-1)+1}=-b_{i+2(m-2)+1}=\cdots =-b_{i+1}$ 
by (3). 
 Hence we get $b_i=-b_{i+1}\cdots $(5). 
\par
 Note that $p$ is odd since $p$ and $q$ are coprime. 
 We set $p=2n+1$ $(n \in {\Bbb N})$. 
 Since we consider indices modulo $p$, 
$b_i=b_{i+p}$ holds.
 (3) implies $b_{i+p}=b_{i+2n+1}=b_{i+2n-1}=b_{i+2n-3}=\cdots =b_{i+1}$. 
 Hence $b_i=b_{i+1}\cdots $(6). 
\par
 By (5) and (6), $b_i=0$ for all $i$. 
\newline
 Thus we have shown that the vectors 
${\bf s}_1, {\bf s}_2, \cdots, {\bf s}_p$,
${\bf t}_1, {\bf t}_2, \cdots, {\bf t}_p$ are 
linearly independent by (i) and (ii). 

 Similarly, we can see that the first $p$ row vectors
of the coefficient matrix of the Q-matching equations 
are linearly independent. 
 Then the dimension of the solution space is $3p-p=2p$. 
 Hence
${\bf s}_1, {\bf s}_2, \cdots, {\bf s}_p$,
${\bf t}_1, {\bf t}_2, \cdots, {\bf t}_p$ form 
a basis of the solution space.


\section{$(p,2)$-lens spaces}

 In this section, 
we prove Theorem \ref{theorem:main} (3),
that is, 
we determine all the Q-fundamental surfaces 
for the triangulation $T(p,2)$ of the $(p,2)$-lens space
with $p \ge 5$. 
 In Section 4, we have obtained a basis of the solution space
($\subset {\Bbb R}^{3t}$)
of the Q-matching equations for the $(p,q)$-lens space. 
 A linear combination of them gives a general solution. 
 For $T(p,2)$, it is
\newline
${\bf v} = a_1 {\bf s}_1 + b_1 {\bf t}_1 + a_2 {\bf s}_2 
+ b_2 {\bf t}_2 + \cdots + a_p{\bf s}_p + b_p {\bf t}_p$
\newline
\hspace*{3mm}
$= 
{}^t (
a_1,\ a_1+b_2+b_{p-1},\ a_1+b_1+b_p\ |\ 
a_2,\ a_2+b_3+b_p,\ a_2+b_2+b_1\ |
\cdots $
\newline
$\ \ \cdots |\ 
a_i,\ a_i+b_{i+1}+b_{i-2},\ a_i+b_i+b_{i-1}\ |
\cdots |\ 
a_p,\ a_p+b_1+b_{p-2},\ a_p+b_p+b_{p-1}\ )$.
\newline
where $a_i, b_i \in {\Bbb R}$.
 We consider when ${\bf v}$ represents a Q-fundamental surface.

\vspace{2mm}
\noindent
{\bf The case where $a_j=0$ for all $j$}: 
\newline
(i) The case where the $3$rd element of the $(i+1)$-st block 
$b_{i+1}+b_i>0\ \cdots$(1) for some $i$. 

 In this case,
\newline
${\bf v} = 
{}^t 
(0,\ b_2+b_{p-1},\ b_1+b_p \ | \cdots |\ 
0,\ b_i+b_{i-3},\ b_{i-1}+b_{i-2} \ |\ 
0,\ b_{i+1}+b_{i-2},\ b_i+b_{i-1} \ |\ $
\newline
\hspace*{1cm}
$0,\ b_{i+2}+b_{i-1},\ b_{i+1}+b_i \ |\ 
0,\ b_{i+3}+b_i,\ b_{i+2}+b_{i+1} \ | \cdots |\ 
0,\ b_1+b_{p-2},\ b_p+b_{p-1})$.

 The $2$nd element of the $(i+1)$-st block $b_{i+2}+b_{i-1}=0\ \cdots$(2)
by the square condition. 
 Adding (2) to (1), we obtain $b_i+b_{i-1}+b_{i+2}+b_{i+1}>0$. 
 Hence either $b_i+b_{i-1}>0$ or $b_{i+2}+b_{i+1}>0$.
 We consider the former case.
 Ultimately we will show that ${\bf v}={\bf t}_i$. 
 (In the latter case, similar argument shows ${\bf v}={\bf t}_{i+1}$, 
and we omit it.)
 Then $b_i+b_{i-1}>0 \ \cdots$(3), 
which is the $3$rd element of the $i$-th block.
 Hence the $2$nd element $b_{i+1}+b_{i-2}=0\ \cdots$(4) by the square conditon. 
 We get $b_{i-1}+b_{i-2}+b_{i+2}+b_{i+1}=0\ \cdots$ (5) from (2) and (4).
 The left hand side of (5) is the sum of 
the $3$rd element of the $(i-1)$-st block $b_{i-1}+b_{i-2}$
and the $3$rd element of the $(i+2)$-nd block $b_{i+2}+b_{i+1}$,
which are both non-negative. 
 Hence they are equal to $0$, i.e.,
$b_{i-1}+b_{i-2}=0\ \cdots$(6) and $b_{i+2}+b_{i+1}=0\ \cdots$(7).
 We assume, for a contradiction, 
that the $2$nd element of the $(i-1)$-st block $b_i+b_{i-3}=0$. 
 This together with (6) gives $b_{i-2}+b_{i-3}+b_i+b_{i-1}=0$,
which is the sum of the $3$rd element of the $(i-2)$-nd block
and the $3$rd element of the $i$-th block.
 Then we obtain $b_i+b_{i-1}=0$, which contradicts (3). 
 Hence we get $b_i+b_{i-3}>0\ \cdots$ (8). 
 Similar argument shows $b_{i+3} + b_i > 0 \ \cdots$ (9).
 (We assume the $2$nd element of the $(i+2)$-nd block $b_{i+3}+b_i=0$, 
to obtain $b_{i+1}+b_i+b_{i+3}+b_{i+2}=0$ by (7). 
 Considering the $3$rd element of the $(i+1)$-th block 
and the $3$rd element of the $(i+3)$-rd block, 
we can see $b_{i+1}+b_i=0$,
contradicting (1).)
\newline
 Then, by (1), (3), (8) and (9),
\newline
${\bf v} =
{}^t (0,\ b_2+b_{p-1},\ b_1+b_p\ | \cdots |\ 
0,\ b_{i-1}+b_{i-4},\ b_{i-2}+b_{i-3}\ |\ 
0,\ b_i+b_{i-3},\ 0\ |\ $
\newline
\hspace*{15mm}
$0,\ 0,\ b_i+b_{i-1}\ |0,\ 0,\ b_{i+1}+b_i\ |
0,\ b_{i+3}+b_i,\ 0 \ |\ 0,\ b_{i+4}+b_{i+1},\ b_{i+3}+b_{i+2}\ | \cdots $
\newline
\hspace*{15mm}
$\cdots |\ 0,\ b_1+b_{p-2},\ b_p+b_{p-1})$
\newline
$\ \ \ge
(0,\ 0,\ 0\ | \cdots |\ 0,\ 0,\ 0\ |\ 0,\ 1,\ 0\ |\ $
\newline
\hspace*{15mm}
$0,\ 0,\ 1\ |\ 0,\ 0,\ 1\ |\ 0,\ 1,\ 0\ |\ 0,\ 0,\ 0\ | \cdots $
\newline
\hspace*{15mm}
$\cdots |\ 0,\ 0,\ 0)={\bf t}_i$. 
\newline
 Since ${\bf v}$ is Q-fundamental, 
${\bf v}={\bf t}_i$, 
which is a candidate of a Q-fundamental solution. 
 We write this vector ${\bf t}''_i$ 
to clarify that it is for a $(p,2)$-lens space.

\vspace{3mm}
\noindent
(ii) The case where $b_i+b_{i+1}=0$ for all $i$.
\newline
 Considering this condition for $i=1,2, \cdots, p$,
we get a system of linear equations.
 Then we solve it to have $b_i=0$ for all $i$,
since $q=2$ implies $p$ is odd.
 Thus ${\bf v} = {\bf 0}$, a contradiction.

\vspace{5mm}
\noindent
{\bf The case where $a_j>0$ for some $j$}: 
\newline
 Since the $1$st element $a_j$ in the $j$-th block is positive,
by the square condition, 
the $3$rd element of the $j$-th block $a_j+b_{j-1}+b_j=0$,
and hence $b_{j-1}+b_j<0$.
 Then
\newline
${\bf v} =
{}^t (a_1,\ a_1+b_2+b_{p-1},\ a_1+b_1+b_p\ |\ 
a_2,\ a_2+b_3+b_p,\ a_2+b_2+b_1\ | \cdots$
\newline
\hspace{15mm}
$\cdots |\ 
a_{j-1},\ a_{j-1}+b_j+b_{j-3},\ a_{j-1}+b_{j-1}+b_{j-2}\ |\ 
a_j,\ 0,\ 0\ |\ $
\newline
\hspace*{15mm}
$a_{j+1},\ a_{j+1}+b_{j+2}+b_{j-1},\ a_{j+1}+b_{j+1}+b_j\ | \cdots |\ 
a_p,\ a_p+b_1+b_{p-2},\ a_p+b_p+b_{p-1})$.


\begin{lemma}\label{lemma:p-2-1}
If $b_{i-1}+b_i<0$ for some $i$, 
then $a_i=-(b_{i-1}+b_i)>0$ and either
$b_{i-2}+b_{i-1}=b_{i-1}+b_i=b_i+b_{i+1}$, 
$b_{i-2}+b_{i-1}<b_{i-1}+b_i<b_i+b_{i+1}$ 
or $b_{i-2}+b_{i-1}>b_{i-1}+b_i>b_i+b_{i+1}$. 
\end{lemma}

\begin{proof}
 Since $b_{i-1}+b_i$ is negative, 
and since the $3$rd element of the $i$-th block $a_i+b_{i-1}+b_i\ge 0$, 
we have $a_i>a_i+b_{i-1}+b_i\ge 0$. 
 In the $i$-th block,
since the $1$st element $a_i>0$, 
the $2$nd and the $3$rd elements are zero by the square condition. 
 Summing up these elements
$a_i+b_{i-1}+b_i=0$ and $a_i+b_{i-2}+b_{i+1}=0$,
we get $2a_i+b_{i-2}+b_{i-1}+b_i+b_{i+1}=0$. 
 Hence $-2a_i=b_{i-2}+b_{i-1}+b_i+b_{i+1}$, 
and we have either $b_{i-2}+b_{i-1}=-a_i=b_i+b_{i+1}$, 
$b_{i-2}+b_{i-1}<-a_i<b_i+b_{i+1}$ 
or $b_{i-2}+b_{i-1}>-a_i>b_i+b_{i+1}$. 

In addition, we obtain $-a_i=b_{i-1}+b_i<0$ from the $3$rd element. 
\end{proof}

\begin{lemma}\label{lemma:p-2-2}
If $b_{l-1}+b_l<0$ for some $l$, 
then $-a_1=-a_2=\cdots =-a_p=b_1+b_2=b_2+b_3=\cdots =b_p+b_1$. 
\end{lemma}

\begin{proof}
 Setting $i=l$ in Lemma \ref{lemma:p-2-1}, 
we have either $b_{l-2}+b_{l-1}=b_{l-1}+b_l=b_l+b_{l+1}$, 
$b_{l-2}+b_{l-1}<b_{l-1}+b_l<b_l+b_{l+1}$ 
or $b_{l-2}+b_{l-1}>b_{l-1}+b_l>b_l+b_{l+1}$.
 Then $b_{l-2}+b_{l-1}\le b_{l-1}+b_l\le b_l+b_{l+1}$ 
or $b_{l-2}+b_{l-1}\ge b_{l-1}+b_l\ge b_l+b_{l+1}$ holds. 
 We consider the case of $b_{l-2}+b_{l-1}\le b_{l-1}+b_l\le b_l+b_{l+1}$. 
 (Similar argument will do for the other case, and we omit it.)
 We will prove 
$0>b_{l-1}+b_l\ge b_{l-2}+b_{l-1}\ge \cdots \ge b_{l-1-n}+b_{l-n}$ 
for any positive integer $n$
by induction. 
 Suppose
$0>b_{l-1}+b_l\ge b_{l-2}+b_{l-1}\ge \cdots \ge b_{l-1-k}+b_{l-k}$. 
 Setting $i=l-k$ in Lemma \ref{lemma:p-2-1}, 
we get either $b_{l-2-k}+b_{l-1-k}=b_{l-1-k}+b_{l-k}=b_{l-k}+b_{l-k+1}$, 
$b_{l-2-k}+b_{l-1-k}>b_{l-1-k}+b_{l-k}>b_{l-k}+b_{l-k+1}$ 
or $b_{l-2-k}+b_{l-1-k}<b_{l-1-k}+b_{l-k}<b_{l-k}+b_{l-k+1}$. 
 The second one
contradicts the assumption of induction. 
 Then we have $0 > b_{l-1-k}+b_{l-k} \ge b_{l-2-k}+b_{l-1-k}$. 
 Hence $0>b_{l-1}+b_l\ge \cdots \ge b_{l-1-n}+b_{l-n}$ 
for all $n \in {\Bbb N}$ by induction. 
 We set $n=p$, 
to obtain $0>b_{l-1}+b_l\ge \cdots \ge b_{l-1-p}+b_{l-p}=b_{l-1}+b_l$, 
where we consider indices modulo $p$.
 Hence $0>b_{l-1}+b_l=\cdots =b_{l-1-p}+b_{l-p}$ follows.

 By Lemma \ref{lemma:p-2-1}, 
$b_{l-1-m}+b_{l-m}=-a_{l-m}$ holds for $0\le m \le p$. 
\end{proof}

\vspace{5mm}

 By Lemma \ref{lemma:p-2-2},
$a_i = -(b_{i-1}+b_i) > 0$ for $\forall i \in \{ 1,2,\cdots, p \}$.
 Hence, 
the $2$nd and the $3$rd elements are zero in each block by the square condition,
and
\newline
${\bf v} =
{}^t (a_1,\ 0,\ 0,\ | \cdots |\ a_j,\ 0,\ 0\ | \cdots |\ a_p,\ 0,\ 0) \ge
{}^t (1,\ 0,\ ,0\ | \cdots |\ 1,\ 0,\ 0\ | \cdots |\ 1,\ 0,\ 0)$
\newline
 Let ${\bf f}''_1$ denote the last vector. 
 Since ${\bf v}$ is Q-fundamental, ${\bf v}={\bf f}''_1$,
which is a candidate of a Q-fundamental solution. 

 Thus we have obtained all the candidates of Q-fundamental solutions 
${\bf t}''_1, \cdots, {\bf t}''_p$ and ${\bf f}''_1$ 
satisfying the square condition. 
 In fact, ${\bf t}''_i$ and ${\bf f}''_1$ are Q-fundamental solutions. 
 ${\bf f}''_1$ is Q-fundamental
because the $1$st elements of the blocks of ${\bf f}''_1$ are $1$
and those of the other vectors ${\bf t}''_1, \cdots, {\bf t}''_p$ are $0$.
 Since size$({\bf t}''_i)=4$, size$({\bf f}''_1)=p$ and $p\ge 5$, 
${\bf t}''_i$ is a Q-fundamental solution. 

\begin{figure}[htbp]
\begin{center}
\includegraphics[width=8cm]{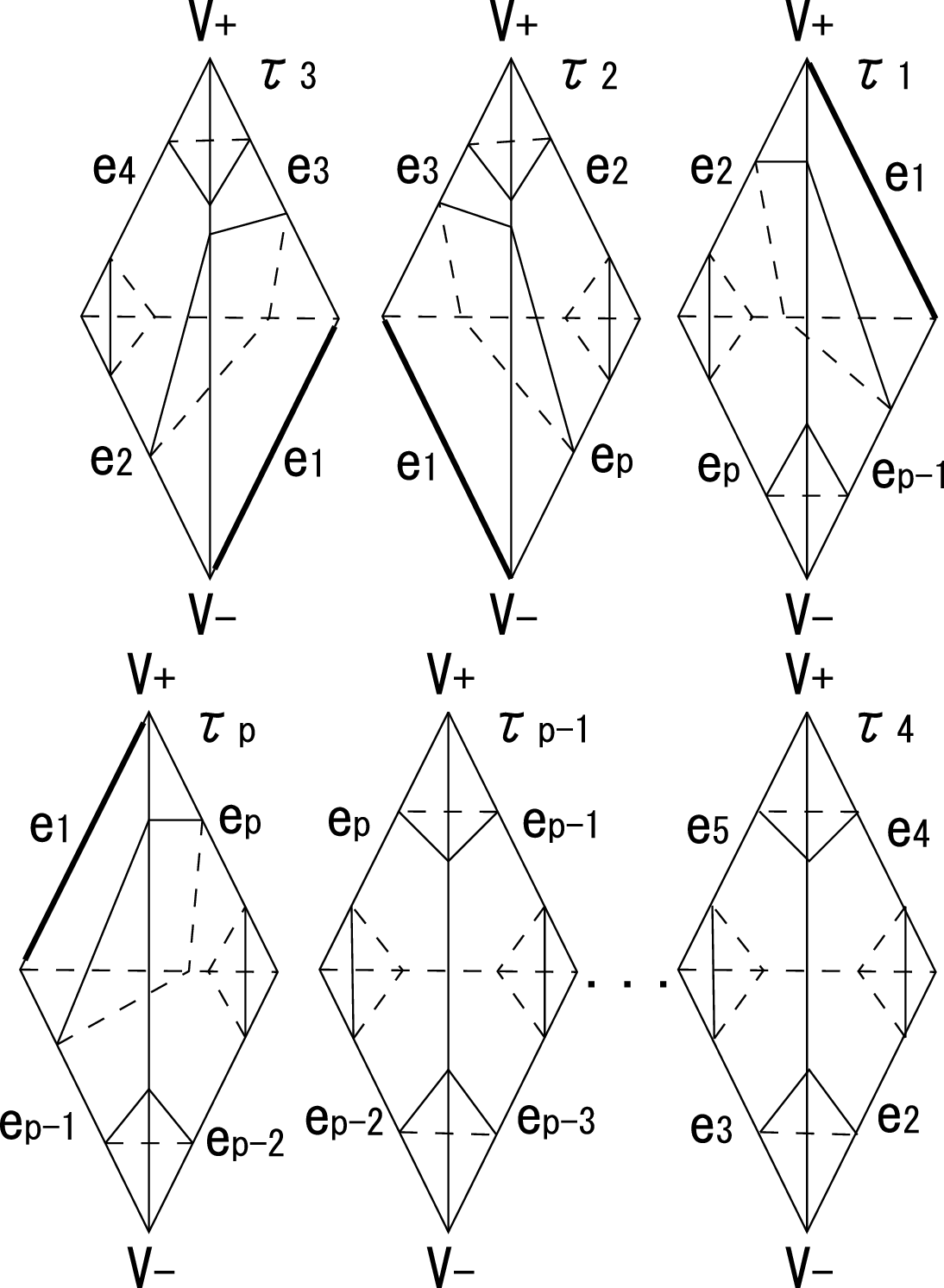}
\end{center}
\caption{}
\label{fig:Lp22-spheresurroundinge}
\end{figure}

 ${\bf t}''_i$ represents the inessential $2$-sphere
surrounding the edge $e_i$.
 See Figure \ref{fig:Lp22-spheresurroundinge}.
 ${\bf f}''_1$ represents a Heegaard splitting torus 
which surrounds the core circle $E_v$.


\section{Q-fundamental surfaces in the $(p,q)$-lens space}

 In this section,
we prove Lemma \ref{lemma:p-qZ} and
Theorems \ref{theorem:p-qfund}, \ref{theorem:p-qfund-podd1/2}.
 In general $(p,q)$-lens spaces,
we consider Q-fundamental surfaces
with no quadrilateral normal disks 
disjoint from the core circles $E_v$ and $E_h$.
 In another words,
we consider 
when 
${\bf v}= a_1{\bf s}_1 + \cdots + a_p{\bf s}_p 
+ b_1{\bf t}_1 + \cdots + b_p{\bf t}_p$
with $a_1 = \cdots = a_p = 0$
represents a Q-fundamental surface.
 We will obtain a restriction on $b_i$'s.

\begin{proof}
 We prove Lemma \ref{lemma:p-qZ}.

 The $j$-th block of ${\bf v}$ is 
${}^t(a_j,\ a_j + b_{j+1}+b_{p-q+j},\ a_j + b_j+b_{p-q+j+1})$
for $1 \le j \le p$.

 Since ${\bf v} \in {\Bbb Z}^{3p}$,
we can see from the $1$st element of each block
that $a_k \in {\Bbb Z}$ for all $k \cdots$ (1).

 Let $i$ be an arbitrary integer with $1\le i\le p$. 
 From the $3$rd element of the $i$-th block,
we can see $a_i+b_i+b_{p-q+i+1}\in {\Bbb Z}$. 
 Then (1) implies $b_i+b_{p-q+i+1}\in {\Bbb Z}\cdots (2)$. 
 On the other hand, 
we have $b_{i+2}+b_{p-q+i+1}\in {\Bbb Z} \cdots$ (3)
from the $2$nd element of the $(i+1)$-st block. 
 By subtract (3) from (2), 
we obtain 
${\Bbb Z}\ni (b_i+b_{p-q+i+1})-(b_{i+2}+b_{p-q+i+1})=b_i-b_{i+2}\cdots (4)$
for all $i$. 

 When $p$ is odd, 
$b_i-b_{i+2},\ b_{i+2}-b_{i+4},\ \cdots,\ b_{i+2(p-2)}-b_{i+2(p-1)} 
\in {\Bbb Z}$ 
by (4), 
and hence ${\mathcal B} \subset {\Bbb Z}+b_1$, 
that is, 
the decimal places of all the $b_i$'s coincide. 
 In particular $b_i-b_{p-q+i+1}\in {\Bbb Z}$.
 This together with (2) implies 
${\Bbb Z} \owns (b_i+b_{p-q+i+1})+(b_i-b_{p-q+i+1})=2b_i\cdots (5)$. 
 Hence $b_i\in {\Bbb Z}$ or $b_i\in {\Bbb Z}+1/2$.

 When $p$ is even, 
${\mathcal B}_0 \in {\Bbb Z}+b_0$ and ${\mathcal B}_1 \in {\Bbb Z}+b_1$ by (4). 
 Since $q$ is odd, $p-q+1$ is even, 
and hence $b_i-b_{p-q+i+1}\in {\Bbb Z}$ for all $i$.
 Then we obtain (5) again, 
and, for $i=0$ and $1$, 
${\mathcal B}_i \subset {\Bbb Z}$ or ${\mathcal B}_i \subset {\Bbb Z}+1/2$.
\end{proof}

 In the rest of this section 
we consider the case where $a_k = 0$ for all $k$. 
 Then ${\bf v}= b_1 {\bf t}_1 + \cdots + b_p {\bf t}_p$, 
$(b_1, \cdots b_p\in {\Bbb R})$.
 The $i$-th block of ${\bf v}$ is 
${}^t(0,\ b_{i+1}+b_{p-q+i},\ b_j+b_{p-q+i+1})$
for $1\le i \le p$.
 Let $\tilde{\bf v}$ be the vector obtained from ${\bf v}$
by deleting the $1$st elements of all the blocks.
 The $1$st and the $2$nd elements of the $i$-th block of $\tilde{\bf v}$ are
equal to the $2$nd and the $3$rd elements of the $i$-th block of ${\bf v}$.

\begin{proof} 
 We prove Theorem \ref{theorem:p-qfund}.


(I)
 Since ${\bf v}$ is Q-fundamental, 
$\tilde{\bf v}$ has at least one positive element $b_j+b_l$.
 Either both $j$ and $l$ are odd, or both even because $p$ is even.
 We suppose they are even, i.e., $j,l \in 2{\Bbb Z}\cdots$ ($*$). 
 (Similar argument will do for the case of $j,l \in 2{\Bbb Z}+1$. We omit it.)
 We set 
${\bf v}'= 
0 {\bf t}_1 + b_2 {\bf t}_2 + \cdots + 0 {\bf t}_{p-1} + b_p {\bf t}_p$,
which is obtained from
${\bf v}= b_1 {\bf t}_1 + \cdots + b_p {\bf t}_p$
by replacing all the $b_i$'s with $i$ odd by $0$.
Then ${\bf v}\ge {\bf v}'$ and ${\bf v}'> {\bf 0}$ by ($*$). 
Since ${\bf v}$ is Q-fundamental, ${\bf v}={\bf v}'$. 
Hence ${\mathcal B}_1 \subset \{ 0\}$. 
${\mathcal B}_0 \subset {\Bbb Z}$ or 
${\mathcal B}_0 \subset {\Bbb Z}+1/2$ by Lemma \ref{lemma:p-qZ}.
 This completes the proof of Theorem \ref{theorem:p-qfund} (I). 

(II)
\noindent
{\bf The case where $p$ is odd}. 
 Set $b'_i=-1/2,0$ or $1/2$ according as $b_i$ is negative, $0$ or positive. 
 Let ${\bf v}''$ be a solution obtained from ${\bf v}$ 
by replacing each $b_i$ with $b'_i$,
that is, ${\bf v}''= b'_1 {\bf t}_1 + \cdots + b'_p {\bf t}_p$.
 For each element $b_i+b_j$ of $\tilde{\bf v}$, 
we will prove $b_i+b_j\ge b'_i+b'_j\ge 0$. 

 In case of $b_i, b_j >0$,
the condition $b_i,b_j\in {\Bbb Z}+1/2$ 
implies $b_i, b_j \ge 1/2$.
 Hence $b_i+b_j\ge 1/2+1/2=b'_i+b'_j$
and $b'_i+b'_j=1/2+1/2=1>0$.

 We consider the case of $b_i<0$ and $b_j>0$ (similar for $b_i>0$ and $b_j<0)$. 
 Since $b_i + b_j$ is the element of ${\bf v}$, 
we have $b_i+b_j\ge 0$. 
 Hence $b_i+b_j\ge 0=-1/2+1/2=b'_i+b'_j$
and $b'_i+b'_j=-1/2+1/2=0 \ge 0$. 

 The condition $b_i<0$ and $b_j<0$ contradicts 
that every element of ${\bf v}$ is non-negative. 

 Now we prove 
that there is an element $b_i+b_j$ of $\tilde{\bf v}$
with $b_i>$ and $b_j > 0$.
 Since $b_i, b_j \in {\Bbb Z} + 1/2$,
each of them is positive or negative.
 Suppose, for a contradiction,
that for every element $b_i + b_j$ of $\tilde{\bf v}$
the signs of $b_i$ and $b_j$ are opposite. 
 Then $sgn(b_i\times b_{i-q+1})=-1\cdots (1)$ 
and $sgn(b_{i+2}\times b_{i-q+1})=-1$ 
from the $3$rd element of the $i$-th block 
and the $2$nd element of the $(i+1)$-st block respectively. 
 Then we obtain 
$sgn(b_i\times b_{i-q+1})\times sgn(b_{i+2}\times b_{i-q+1})=1$, 
that is, $sgn(b_i\times b_{i+2})=1\cdots (2)$ for all $i$. 
 Since $p$ is odd, 
by applying (2) repeatedly
sgn\,$(b_1)=$\,sgn\,$(b_3)= \cdots =$\,sgn\,$(b_p)=$\,sgn\,$(b_2)=$\,
sgn\,$(b_4)= \cdots =$\,sgn\,$(b_{p-1})$.
 Thus the signs of all the elements of ${\mathcal B}$ coincide.
 In particular, $sgn(b_i\times b_{i-q+1})=1$,
which contradicts $(1)$.
 Hence there is an element $b_i+b_j$ of $\tilde{\bf v}$
with $b_i, b_j>0$.
 Then ${\bf v}''$ has an element $b'_i+b'_j > 0$
since $sgn(b'_i)=sgn(b_i)$ and $sgn(b'_j)=sgn(b_j)$ by definition.

 Thus ${\bf v} \ge {\bf v}''>{\bf 0}$. 
 Since ${\bf v}''$ is a solution and ${\bf v}$ is Q-fundamental, 
${\bf v}={\bf v}''$.

\noindent
{\bf The case where $p$ is even}. 
 We can assume without loss of generality
that ${\mathcal B}_0 = \{ 0 \}$ and ${\mathcal B}_1 \subset {\Bbb Z} + 1/2$.
 For each element $b_i + b_j$ of $\tilde{\bf v}$, 
$i$ and $j$ are both even or both odd.

 Similar arguments as in the previous case will do
except the followings. 

 In case of $b_i=0$ and $b_j=0$ (with $i, j$ even), 
we obtain $b_i+b_j=0+0=b'_i+b'_j$. 
 Hence $b_i+b_j\ge b'_i+b'_j\ge 0$ holds. 

 We prove 
that $\tilde{\bf v}$ has an element $b_k+b_l$ with $b_k, b_l>0$.
 For all the elements $b_k+b_l$ with $k,l$ odd, 
suppose that $b_k+b_l$ are of the opposite signs. 
 For $\forall i \in {\Bbb Z}$
we have $(1)$ and $(2)$ again. 
 Since $q$ is odd,
$-q+1$ is even,
and by applying (2) repeatedly
we obtain $sgn(b_i\times b_{i-q+1})=1$
for all odd integer $i$ with $1 \le i \le p$.
 This contradicts (1).
 Thus the proof of Theorem \ref{theorem:p-qfund} (II) is completed.

\vspace{2mm}

(III)
 Set $b^*_i=-1,0$ or $1$ according as $b_i$ is negative, $0$ or positive. 
 Let ${\bf v}^*$ be the vector 
obtained from ${\bf v}$
by replacing each $b_i$ with $b^*_i$. 

\noindent
{\bf The case where $p$ is odd.} 
 For all the elements $b_i+b_j$ of $\tilde{\bf v}$, 
we will prove $b_i+b_j\ge b^*_i+b^*_j\ge 0$ for all $i,j$. 

 If $b_i>0$ and $b_j>0$, 
then $b_i\ge 1$ and $b_j\ge 1$ because $b_i,b_j\in {\Bbb Z}$. 
 Hence $b_i+b_j\ge 1+1=b^*_i+b^*_j>0$. 

 We consider the case of $b_i<0$ and $b_j>0$.
 (The same argument will do for the case of $b_i>0$ and $b_j<0$.)
 Since $b_i+b_j$ is an element of ${\bf v}$, it is non-negative.
 Hence we have $b_i+b_j \ge 0=-1+1=b^*_i+b^*_j\ge 0$. 

 When $b_i=0$ and $b_j>0$ 
(the proof is the same for the case of $b_i>0$ and $b_j=0$), 
$b_j\in {\Bbb Z}$ implies $b_j\ge 1$,
and hence $b_i+b_j\ge 0+1=b^*_i+b^*_j>0$. 

 In case of $b_i=0$ and $b_j=0$,
clearly $b_i+b_j\ge b^*_i+b^*_j\ge 0$
because $b_i+b_j=0+0=b^*_i+b^*_j$.


 We prove that $\tilde{\bf v}$ has 
an element $b_i+b_j$ 
with ($b_i\ge 0$ and $b_j> 0$) or ($b_i>0$ and $b_j\ge 0$). 
 Suppose not.
 Then, 
for each element $b_i+ b_j$ of $\tilde{\bf v}$, 
either (1) $b_i=b_j=0$ or (2) $b_i$ and $b_j$ are of opposite signs.

 Suppose $b_k = 0$ for some $k$. 
 Then (1) holds rather than (2) 
for the $2$nd element $b_k+b_{k-q+1}$ of the $k$-th block of $\tilde{\bf v}$,
that is, $b_{k-q+1}=0$.
 This implies that 
the $1$st element $b_{k+2}+b_{k-q+1}$ of the $(k+1)$-st block of $\tilde{\bf v}$
satisfies (1). 
 Then we have $b_{k+2}=0$.
 Hence $b_k=0$ implies $b_{k+2}=0$ for all $k \cdots (3)$. 
 Since $p$ is odd, $b_k=0$ for all $k$. 
 Then $\tilde{\bf v} = {\bf 0}$,
and hence ${\bf v}={\bf 0}$, contradicting that it is Q-fundamental. 
 Hence $b_k\not=0$ for all $k$. 

 Thus we can assume that, for all the elements $b_i+b_j$ of $\tilde{\bf v}$,
(2) holds, i.e., $b_i$ and $b_j$ are of opposite signs. 
 Then we obtain a contradiction 
by similar arguments as in the former half of (II).
 Hence $\tilde{\bf v}$ has 
an element $b_i+b_j$ 
with ($b_i\ge 0$ and $b_j> 0$) or ($b_i>0$ and $b_j\ge 0$). 

 Thus ${\bf v}\ge {\bf v}^* >0$. 
 Since ${\bf v}$ is a Q-fundamental, ${\bf v}={\bf v}^*$. 

\noindent
{\bf The case where $p$ is even.}
 The proof is similar to the previous case,
and to the latter half of (II).
 We omit it.
 This completes the proof of Theorem \ref{theorem:p-qfund} (III).
\end{proof}

\vspace{3mm}

\noindent
\begin{proof}
 We prove Theorem \ref{theorem:p-qfund-podd1/2}.
 Suppose, for a contradiction, that ${\bf v}$ is Q-fundamental. 
 Then $b_l=1/2$ or $-1/2$ for all $l$ by (II) in Theorem \ref{theorem:p-qfund}. 
 Note that an element $b_i+b_j$ with $b_i=-1/2$ and $b_j=-1/2$ cannot exist.
 In the expression
\newline
$\tilde{\bf v}
=
{}^t 
(b_2+b_{p-q+1},\ b_1+b_{p-q+2}\ |\ b_3+b_{p-q+2},\ b_2+b_{p-q+3}\ | \cdots$
\newline
\hspace*{15mm}
$\cdots |\ b_{i+1}+b_{p-q+i},\ b_i+b_{p-q+i+1}\ | \cdots |\ 
b_1+b_{p-q},\ b_p+b_{p-q+1})$,
\newline
each $b_i$ appears four times. 
 An element $b_i+b_j$ is equal to $0$
if and only if precisely one of $b_i$ and $b_j$ is equal to $-1/2$.
 Thus
the number of elements equal to $0$ is 
(the number of $b_i$'s equal to $-1/2) \times 4$. 
 In particular it is even $\cdots (*)$.

 The square condition implies 
that each block of $\tilde{\bf v}$ has one or two elements equal to $0$.
 Let $m$ be the number of blocks with their two elements equal to $0$.

 We will prove that 
the blocks with their two elements equal to $0$
separate into pairs.
 Suppose that the block for $\tau_i$ has the two elements equal to $0$.
 There are four patterns below.
\begin{enumerate}
\item[(i)] When $b_i=-1/2$ and $b_{i+1}=-1/2$,
 the pair of blocks for $\tau_i$ and $\tau_{i+q}$ are ${\bf 0}$. 
\item[(ii)] When $b_{i-q}=-1/2$ and $b_i=-1/2$, 
 the pair of blocks for $\tau_i$ and $\tau_{i-1}$ are ${\bf 0}$. 
\item[(iii)] When $b_{i-q+1}=-1/2$ and $b_{i+1}=-1/2$, 
 the pair of blocks for $\tau_i$ and $\tau_{i+1}$ are ${\bf 0}$. 
\item[(iv)] When $b_{i-q}=-1/2$ and $b_{i-q+1}=-1/2$, 
 the pair of blocks for $\tau_i$ and $\tau_{i-q}$ are ${\bf 0}$. 
\end{enumerate}
 Note that the block for $\tau_i$ is
$\left( \begin{array}{c}
b_{i+1}   + b_{i-q} \\
b_i + b_{i-q+1} \\
\end{array} \right)$, 
while the blocks for $\tau_{i-1}$, 
$\tau_{i+1}$, $\tau_{i+q}$ and $\tau_{i-q}$ 
are
$\left( \begin{array}{c}
b_i + b_{i-q-1} \\
b_{i-1}   + b_{i-q} \\
\end{array} \right)$, 
$\left( \begin{array}{c}
b_{i+2} + b_{i-q+1} \\
b_{i+1} + b_{i-q+2} \\
\end{array} \right)$, 
$\left( \begin{array}{c}
b_{i+q+1} + b_i \\
b_{i+q}   + b_{i+1} \\
\end{array} \right)$
and
$\left( \begin{array}{c}
b_{i-q+1}  + b_{i-2q} \\
b_{i-q} + b_{i-2q+1} \\
\end{array} \right)$
respectively.
 Precisely one block has the same $b_j$'s equal to $-1/2$ as that for $\tau_i$.

 Hence $m=2k$ for some non-negative integer $k$,
and the number of elements equal to $0$ of $\tilde{\bf v}$ 
is $2m+1\times(p-m)=2\times 2k+1\times (p-2k)=p+2k$. 
 Since $p$ is odd, $p+2k$ is odd. 
 This contradicts $(*)$. 
\end{proof}


\section{Examples of fundamental surfaces}

 In this section, 
several examples of fundamental surfaces in lens spaces are given.
 The authors have comfirmed 
that they are actually Q-fundamental by computer
except for the last one in the $(418,153)$-lens space.
 By Lemma \ref{lemma:HakenFund},
they are fundamental with respect to Haken's matching equations
except the inessential torus in the $(18,7)$-lens space.
 We begin this section with showing the lemma.

 Recall that
a closed surface $F$ in a lens space
intersects each of $E_v$ and $E_h$ 
in odd number of points
if and only if $F$ is non-orientable.


\begin{proof}
 We prove Lemma \ref{lemma:HakenFund}.

 Suppose, for a contradiction,
that $F$ is decomposed as $F = F_1 + F_2$.
 Since $F$ intersects $E_v$ and $E_h$ in a single point,
one of $F_1$ or $F_2$, say $F_1$ intersects $E_v$ and $E_h$ in a single point,
and $F_2$ is disjoint from $E_v \cup E_h$.
 Then $F_2$ is ${}^t(1,0,0\ |\ 1,0,0\ | \cdots |\ 1,0,0)$,
the Heegaard splitting torus surrounding $E_v$ and $E_h$,
which is the only normal surface disjoint from $E_v \cup E_h$.
 Note that $F_2$ has a normal disks of type $X_{k1}$ 
in each tetrahedron $\tau_k$.
 By assumption, $F$ has a normal disk of type $X_{k2}$ or $X_{k3}$
which cannot exist together with a normal disk of $X_{k1}$.
 This is a contradiction.
\end{proof}

 We consider the $(p,q)$-lens space
with $p$ even and $q \ge 3$.
 Then
${\bf h} = (\sum_{k=1}^{p/2} {\bf t}_{2k-1})/2 = 
{}^t (0 \ 0 \ 1 | 0 \ 1 \ 0 | 0 \ 0 \ 1 | 0 \ 1 \ 0 |
\cdots | 0 \ 0 \ 1 | 0 \ 1 \ 0 )$ is fundamental and not Q-fundamental.
 In fact, it is geometrically compressible. 
 (For the definition, see section 2.)

\begin{figure}[htbp]
\begin{center}
\includegraphics[width=5cm]{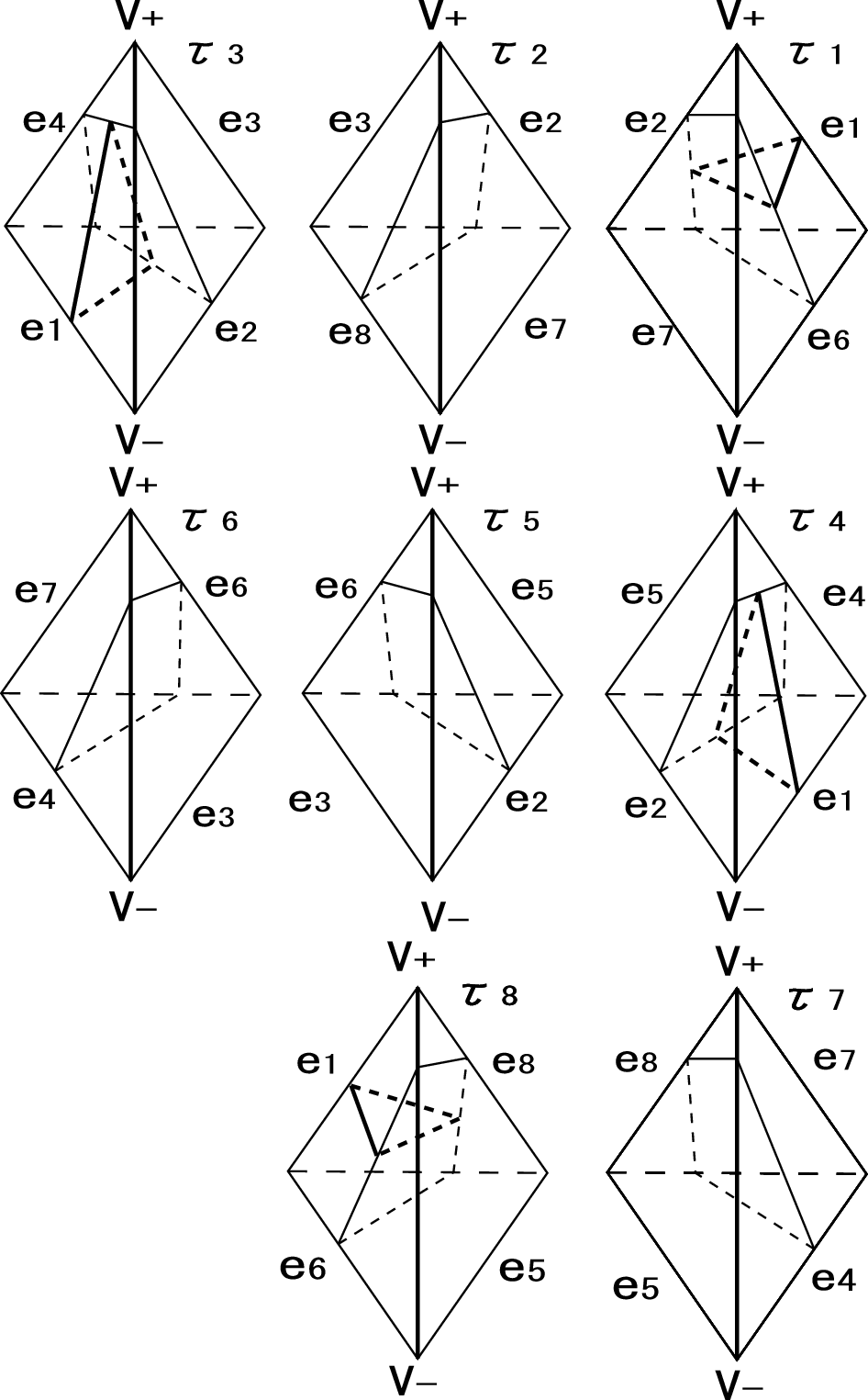}
\end{center}
\caption{}
\label{fig:ComprDisk8-3}
\end{figure}

 For example, 
in the $(8,3)$-lens space,
the normal surface 
\newline
${\bf h} = {}^t (0 \ 0 \ 1 | 0 \ 1 \ 0 | 0 \ 0 \ 1 | 0 \ 1 \ 0 | 
0 \ 0 \ 1 | 0 \ 1 \ 0 | 0 \ 0 \ 1 | 0 \ 1 \ 0 )$
is not Q-fundamental. 
 It has a compressing disk as shown in Figure \ref{fig:ComprDisk8-3},
and compressing yields the Q-fundamental surface
${\bf h} - {\bf t}_1 = 
{}^t (0 \ 0 \ 0 | 0 \ 1 \ 0 | 0 \ 0 \ 0 | 0 \ 0 \ 0 | 
0 \ 0 \ 1 | 0 \ 1 \ 0 | 0 \ 0 \ 1 | 0 \ 0 \ 0)$
which represents
a non-orientable closed surface with maximal Euler characteristic,
a Klein bottle in this case.
 See Figure \ref{fig:83lens}.

\begin{figure}[htbp]
\begin{center}
\includegraphics[width=5cm]{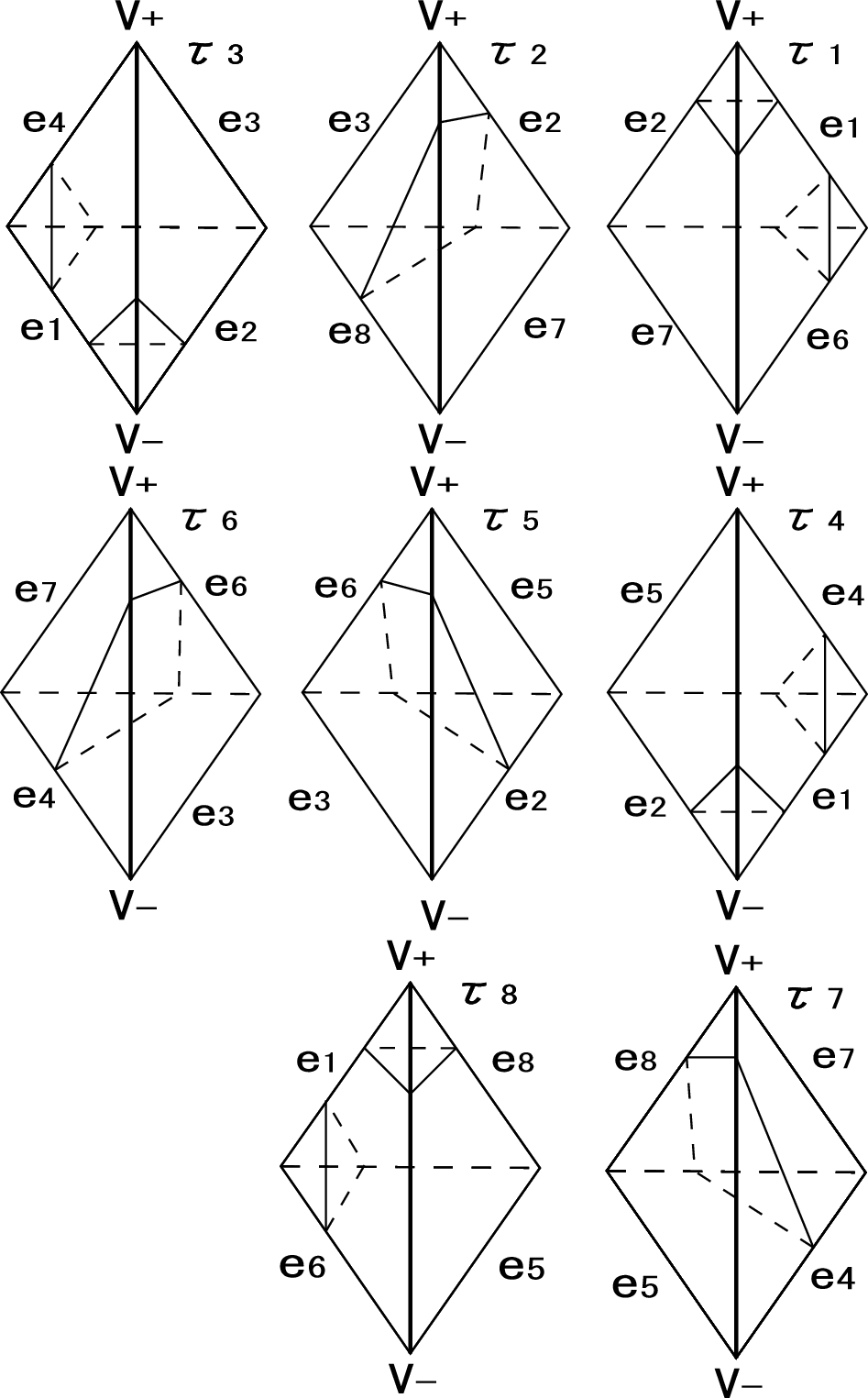}
\end{center}
\caption{}
\label{fig:83lens}
\end{figure}

 In case of $(16,3)$-lens space,
there are two ways of performing compressing operations twice.
 Both
\newline
${\bf h} - {\bf t}_1 - {\bf t}_9$
\newline
$ = 
{}^t (0 \ 0 \ 0 | 0 \ 1 \ 0 | 0 \ 0 \ 0 | 0 \ 0 \ 0 | 
0 \ 0 \ 1 | 0 \ 1 \ 0 | 0 \ 0 \ 1 | 0 \ 0 \ 0 | 
0 \ 0 \ 0 | 0 \ 1 \ 0 | 0 \ 0 \ 0 | 0 \ 0 \ 0 | 
0 \ 0 \ 1 | 0 \ 1 \ 0 | 0 \ 0 \ 1 | 0 \ 0 \ 0)$
and 
\newline
${\bf h} - {\bf t}_1 - {\bf t}_7$
\newline
$ = 
{}^t (0 \ 0 \ 0 | 0 \ 1 \ 0 | 0 \ 0 \ 0 | 0 \ 0 \ 0 | 
0 \ 0 \ 1 | 0 \ 0 \ 0 | 0 \ 0 \ 0 | 0 \ 1 \ 0 | 
0 \ 0 \ 0 | 0 \ 0 \ 0 | 0 \ 0 \ 1 | 0 \ 1 \ 0 | 
0 \ 0 \ 1 | 0 \ 1 \ 0 | 0 \ 0 \ 1 | 0 \ 0 \ 0)$
\newline
are Q-fundamental surfaces
which represent 
non-orientable closed surfaces with maximal Euler characteristics, 
 In the $(16,3)$-lens space,
such a surface is the connected sum of four projective planes
as shown in \cite{BW}.

 There is a Q-fundamental surface with some $a_k$ non-zero.
 For example, $(18,7)$-lens space has such one with the coordinate
\newline
${}^t (0 \ 0 \ 1 | 0 \ 0 \ 0 | 0 \ 0 \ 0 | 0 \ 1 \ 0 |
0 \ 0 \ 1 | 0 \ 0 \ 0 | 0 \ 1 \ 0 | 0 \ 0 \ 0 |
1 \ 0 \ 0 | 1 \ 0 \ 0 | 0 \ 0 \ 0 | 0 \ 1 \ 0 |$
\newline
\hspace*{3cm}
$0 \ 0 \ 0 | 0 \ 0 \ 1 | 0 \ 1 \ 0 | 0 \ 0 \ 0 |
0 \ 0 \ 0 | 0 \ 0 \ 1)$,
\newline
which represents a compressible torus.

\begin{figure}[htbp]
\begin{center}
\includegraphics[width=10cm]{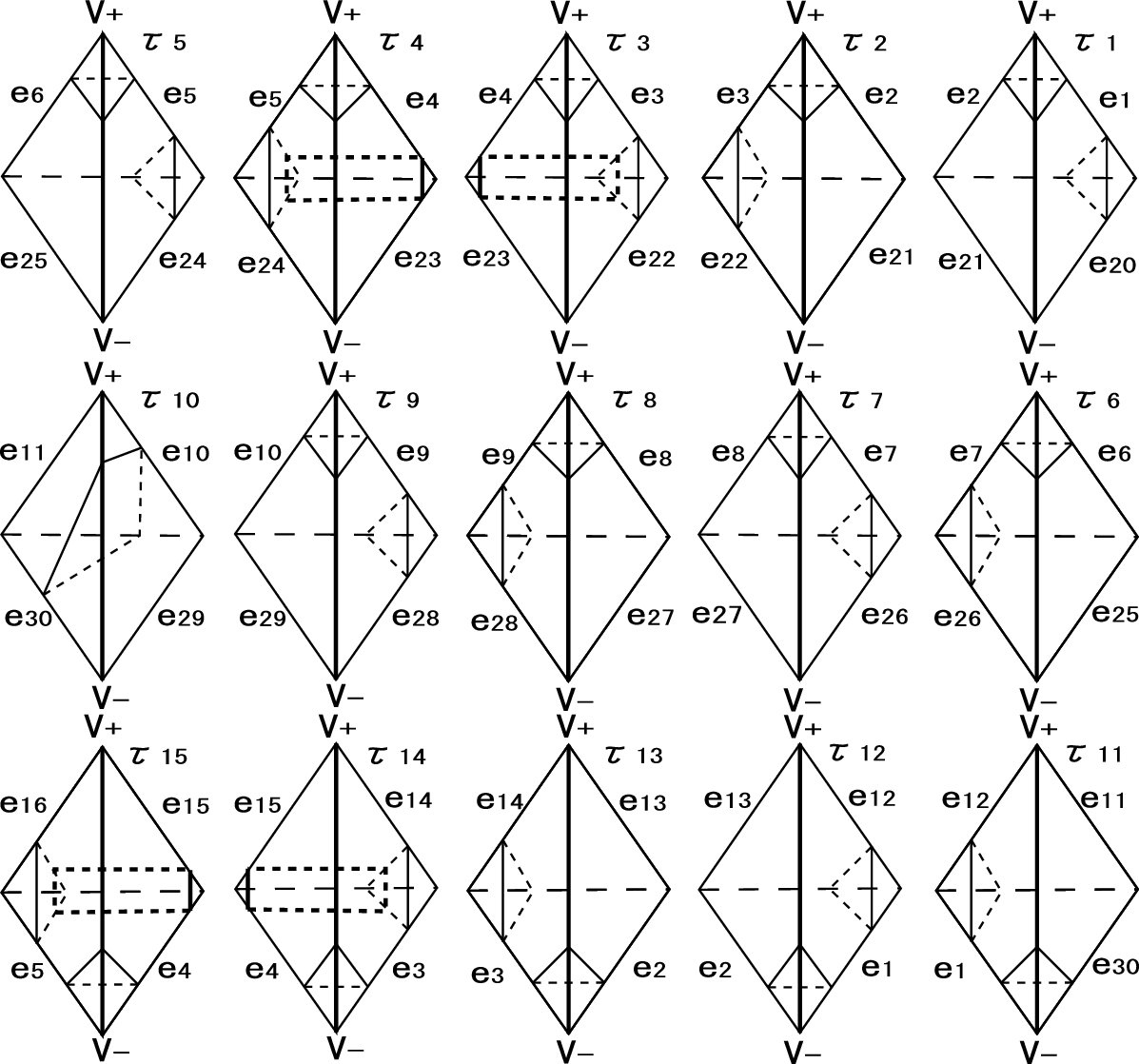}
\end{center}
\caption{}
\label{fig:ComprDisk30-11}
\end{figure}

\begin{figure}[htbp]
\begin{center}
\includegraphics[width=10cm]{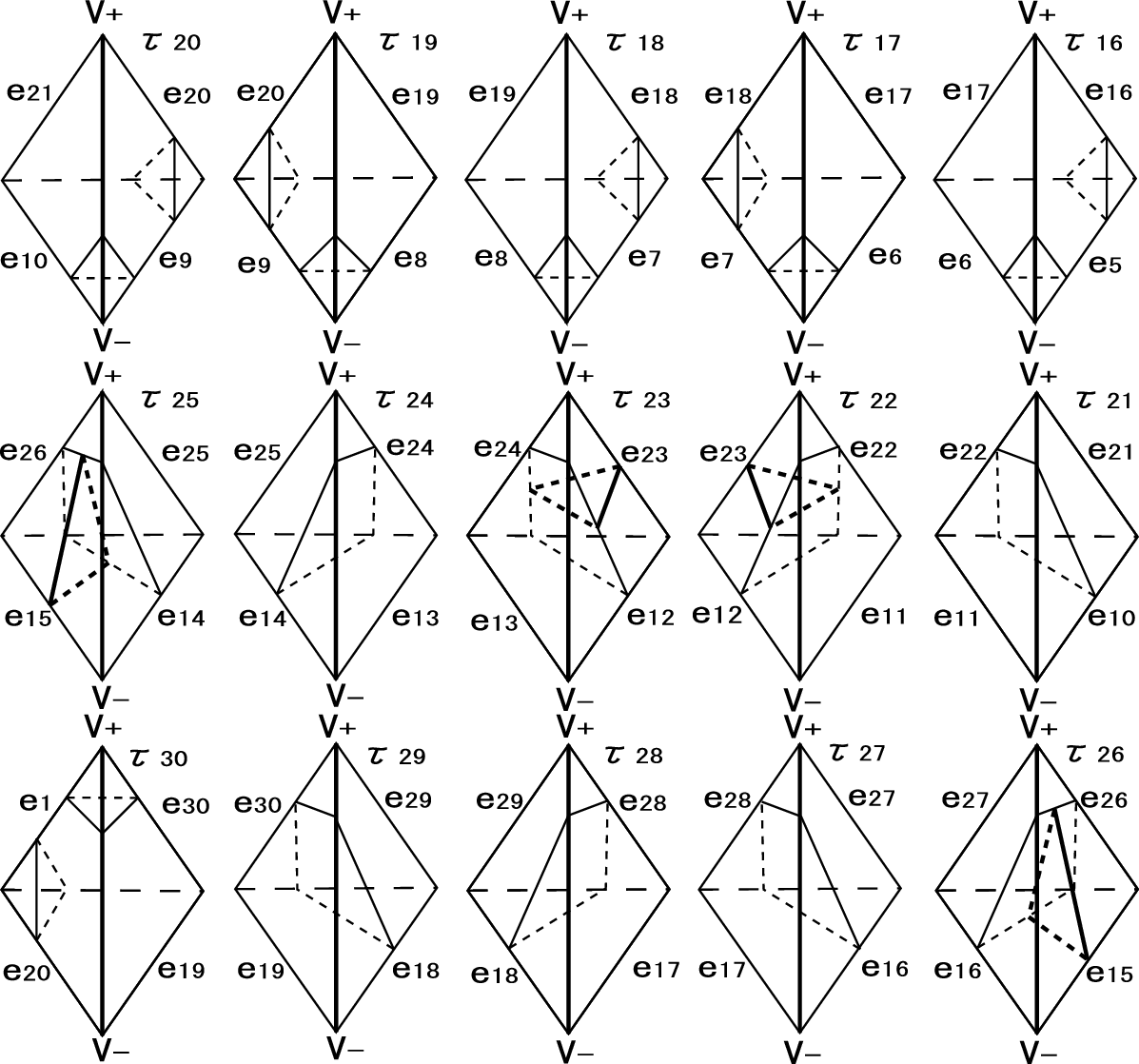}
\end{center}
\caption{}
\label{fig:ComprDisk30-11b}
\end{figure}
\begin{figure}[htbp]

\begin{center}
\includegraphics[width=10cm]{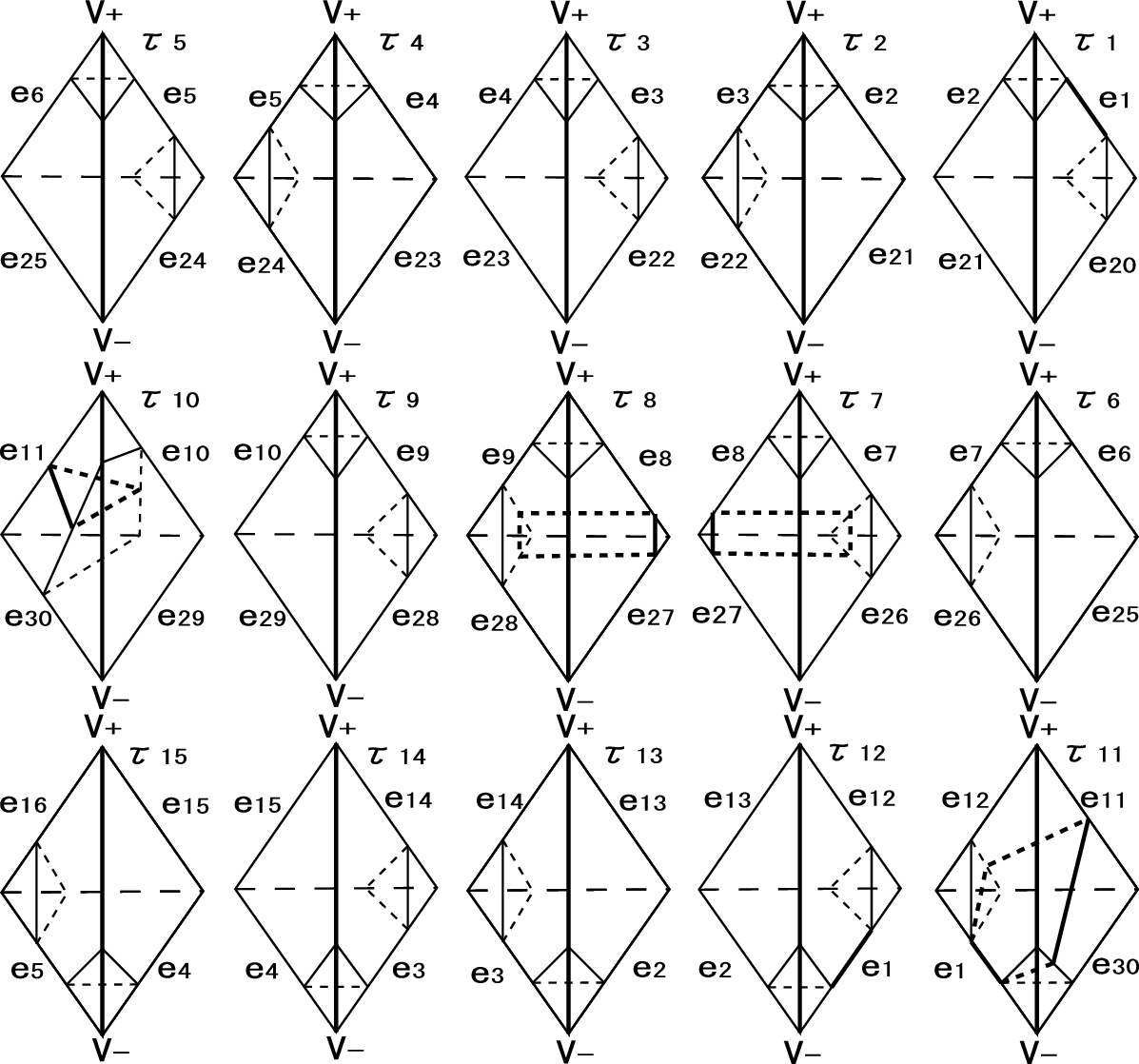}
\end{center}
\caption{}
\label{fig:ComprDisk30-11c}
\end{figure}

\begin{figure}[htbp]
\begin{center}
\includegraphics[width=10cm]{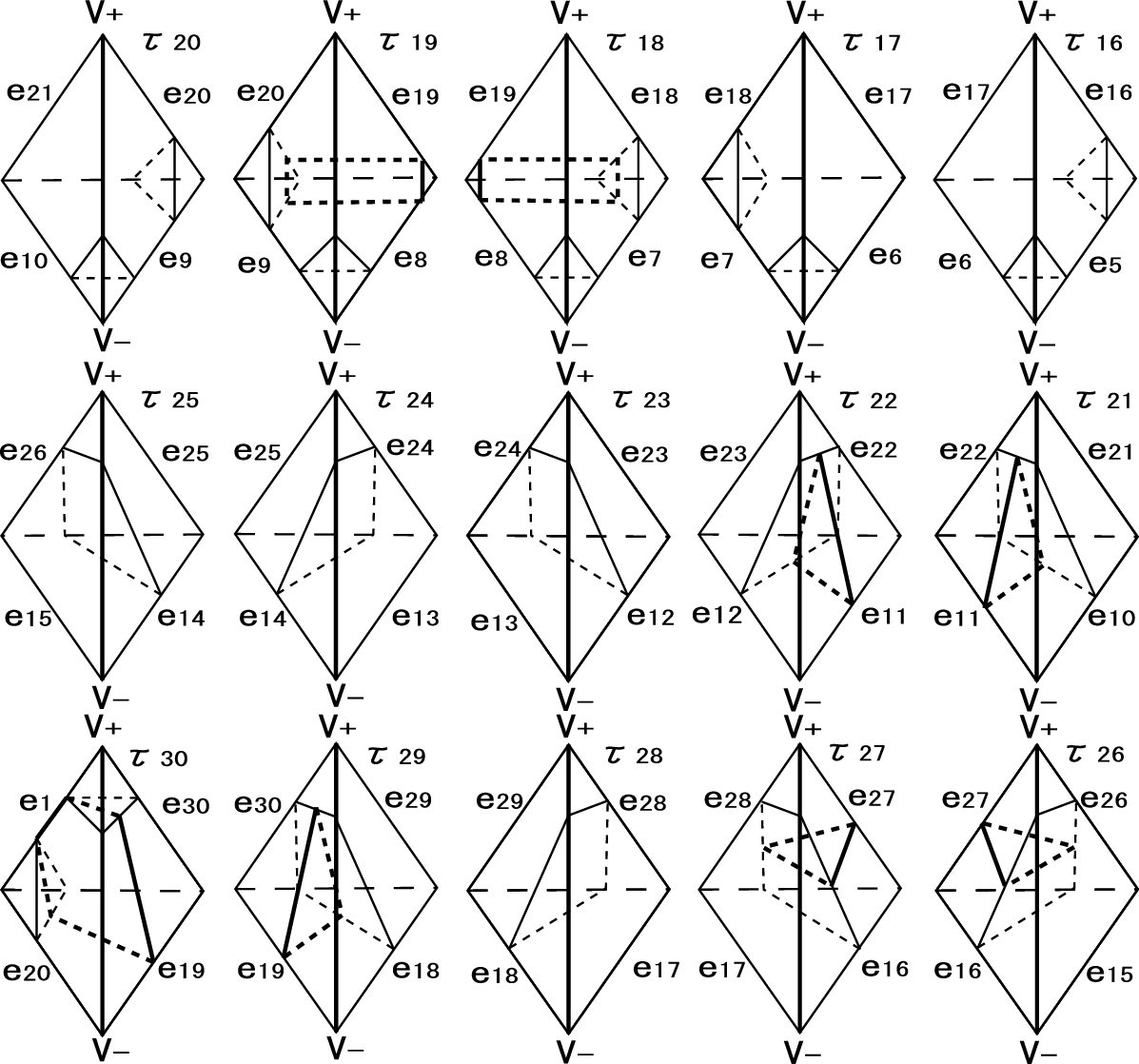}
\end{center}
\caption{}
\label{fig:ComprDisk30-11d}
\end{figure}

 We consider the $(30,11)$-lens space,
where ${\bf h}$ has five compressing disks,
and compression along them yields
the normal surface
\newline
${\bf h} - {\bf t}_1 - {\bf t}_3 - {\bf t}_5 - {\bf t}_7 -{\bf t}_9$
\newline
$= 
{}^t (0 \ 0 \ 0 | 0 \ 0 \ 0 | 0 \ 0 \ 0 | 0 \ 0 \ 0 | 0 \ 0 \ 0 |
0 \ 0 \ 0 | 0 \ 0 \ 0 | 0 \ 0 \ 0 | 0 \ 0 \ 0 | 0 \ 0 \ 1 |
0 \ 0 \ 0 | 0 \ 0 \ 0 | 0 \ 0 \ 0 | 0 \ 0 \ 0 | 0 \ 0 \ 0 |$
\newline
\hspace*{1cm}
$0 \ 0 \ 0 | 0 \ 0 \ 0 | 0 \ 0 \ 0 | 0 \ 0 \ 0 | 0 \ 0 \ 0 |
0 \ 0 \ 1 | 0 \ 1 \ 0 | 0 \ 0 \ 1 | 0 \ 1 \ 0 | 0 \ 0 \ 1 |
0 \ 1 \ 0 | 0 \ 0 \ 1 | 0 \ 1 \ 0 | 0 \ 0 \ 1 | 0 \ 0 \ 0 )$.
\newline
 This represents a non-orientable surface
which is the connected sum of five projective planes.
 This Q-fundamental surface is, however,
not of maximal Euler characteristic.
 In fact, this surface admits a compressing disk 
as in Figures \ref{fig:ComprDisk30-11} and \ref{fig:ComprDisk30-11b}.
 Surgering along this disk,
we obtain the Q-fundamental surface
\newline
${}^t (0 \ 0 \ 0 | 0 \ 0 \ 0 | 1 \ 0 \ 0 | 1 \ 0 \ 0 | 0 \ 0 \ 0 |
0 \ 0 \ 0 | 0 \ 0 \ 0 | 0 \ 0 \ 0 | 0 \ 0 \ 0 | 0 \ 0 \ 1 |
0 \ 0 \ 0 | 0 \ 0 \ 0 | 0 \ 0 \ 0 | 1 \ 0 \ 0 | 1 \ 0 \ 0 |$
\newline
$0 \ 0 \ 0 | 0 \ 0 \ 0 | 0 \ 0 \ 0 | 0 \ 0 \ 0 | 0 \ 0 \ 0 |
0 \ 0 \ 1 | 0 \ 0 \ 0 | 0 \ 0 \ 0 | 0 \ 1 \ 0 | 0 \ 0 \ 0 |
0 \ 0 \ 0 | 0 \ 0 \ 1 | 0 \ 1 \ 0 | 0 \ 0 \ 1 | 0 \ 0 \ 0 )$
\newline
with four $a_k$'s non-zero.
 This surface is the connected sum of three projective planes
and of maximal Euler characteristic
as in \cite{BW}.
 Such a surface is also obtained
by another compression along the disk 
in Figures \ref{fig:ComprDisk30-11c} and \ref{fig:ComprDisk30-11d}.
 The result is the Q-fundamental surface
\newline
${}^t (0 \ 0 \ 1 | 0 \ 0 \ 0 | 0 \ 0 \ 0 | 0 \ 0 \ 0 | 0 \ 0 \ 0 |
0 \ 0 \ 0 | 1 \ 0 \ 0 | 1 \ 0 \ 0 | 0 \ 0 \ 0 | 0 \ 0 \ 0 |
0 \ 0 \ 0 | 0 \ 1 \ 0 | 0 \ 0 \ 0 | 0 \ 0 \ 0 | 0 \ 0 \ 0 |$
\newline
$0 \ 0 \ 0 | 0 \ 0 \ 0 | 1 \ 0 \ 0 | 1 \ 0 \ 0 | 0 \ 0 \ 0 |
0 \ 0 \ 0 | 0 \ 0 \ 0 | 0 \ 0 \ 1 | 0 \ 1 \ 0 | 0 \ 0 \ 1 |
0 \ 0 \ 0 | 0 \ 0 \ 0 | 0 \ 1 \ 0 | 0 \ 0 \ 0 | 0 \ 0 \ 0 )$.

 The $(418, 153)$-lens space contains
a fundamental surface as below
which is homeomorphis to the connected sum of five projective planes
and has three sheets of normal disks of type $X_{k1}$ for four $k$'s.
\newpage
\noindent
${}^t(0, 0, 0 |
  0, 0, 0|
  0, 0, 0|
  1, 0, 0|
  1, 0, 0|
  0, 0, 0|
  0, 0, 0|
  0, 0, 0|
  0, 0, 0|
  0, 0, 0|
  0, 0, 0|
  1, 0, 0|
  1, 0, 0|
  1, 0, 0|
  2, 0, 0|$
\newline 
$2, 0, 0|
  1, 0, 0|
  1, 0, 0|
  1, 0, 0|
  1, 0, 0|
  1, 0, 0|
  0, 0, 0|
  0, 0, 0|
  0, 0, 0|
  0, 0, 0|
  1, 0, 0|
  1, 0, 0|
  0, 0, 0|
  0, 0, 0|
  0, 0, 0|$
\newline
$0, 0, 0|
  0, 0, 0|
  0, 0, 0|
  1, 0, 0|
  1, 0, 0|
  0, 0, 0|
  0, 0, 0|
  0, 0, 0|
  0, 0, 0|
  0, 0, 0|
  0, 0, 0|
  1, 0, 0|
  1, 0, 0|
  1, 0, 0|
  2, 0, 0|$
\newline
$2, 0, 0|
  1, 0, 0|
  1, 0, 0|
  1, 0, 0|
  1, 0, 0|
  1, 0, 0|
  1, 0, 0|
  2, 0, 0|
  2, 0, 0|
  2, 0, 0|
  3, 0, 0|
  3, 0, 0|
  2, 0, 0|
  2, 0, 0|
  2, 0, 0|$
\newline
$2, 0, 0|
  2, 0, 0|
  1, 0, 0|
  1, 0, 0|
  1, 0, 0|
  1, 0, 0|
  2, 0, 0|
  2, 0, 0|
  1, 0, 0|
  1, 0, 0|
  1, 0, 0|
  1, 0, 0|
  1, 0, 0|
  1, 0, 0|
  2, 0, 0|$
\newline
$2, 0, 0|
  1, 0, 0|
  1, 0, 0|
  1, 0, 0|
  1, 0, 0|
  1, 0, 0|
  0, 0, 0|
  0, 0, 0|
  0, 0, 0|
  0, 0, 0|
  1, 0, 0|
  1, 0, 0|
  0, 0, 0|
  0, 0, 0|
  0, 0, 0|$
\newline
$0, 0, 0|
  0, 0, 0|
  0, 0, 0|
  1, 0, 0|
  1, 0, 0|
  1, 0, 0|
  2, 0, 0|
  2, 0, 0|
  1, 0, 0|
  1, 0, 0|
  1, 0, 0|
  1, 0, 0|
  1, 0, 0|
  0, 0, 0|
  0, 0, 0|$
\newline
$0, 0, 0|
  0, 0, 0|
  1, 0, 0|
  1, 0, 0|
  0, 0, 0|
  0, 0, 0|
  0, 0, 0|
  0, 0, 0|
  0, 0, 0|
  0, 0, 0|
  1, 0, 0|
  1, 0, 0|
  0, 0, 0|
  0, 0, 0|
  0, 0, 0|$
\newline
$0, 0, 0|
  0, 0, 0|
  0, 0, 0|
  1, 0, 0|
  1, 0, 0|
  1, 0, 0|
  2, 0, 0|
  2, 0, 0|
  1, 0, 0|
  1, 0, 0|
  1, 0, 0|
  1, 0, 0|
  1, 0, 0|
  0, 0, 0|
  0, 0, 0|$
\newline
$0, 0, 0|
  0, 0, 0|
  1, 0, 0|
  1, 0, 0|
  0, 0, 0|
  0, 0, 0|
  0, 0, 0|
  0, 0, 0|
  0, 0, 0|
  0, 0, 0|
  1, 0, 0|
  1, 0, 0|
  0, 0, 0|
  0, 0, 0|
  0, 0, 0|$
\newline
$0, 0, 0|
  0, 0, 0|
  0, 1, 0|
  0, 0, 0|
  0, 0, 0|
  0, 0, 0|
  1, 0, 0|
  1, 0, 0|
  0, 0, 0|
  0, 0, 0|
  0, 0, 0|
  0, 0, 0|
  0, 0, 0|
  0, 0, 0|
  1, 0, 0|$
\newline
$1, 0, 0|
  1, 0, 0|
  2, 0, 0|
  2, 0, 0|
  1, 0, 0|
  1, 0, 0|
  1, 0, 0|
  1, 0, 0|
  1, 0, 0|
  0, 0, 0|
  0, 0, 0|
  0, 0, 0|
  0, 0, 0|
  1, 0, 0|
  1, 0, 0|$
\newline
$ 0, 0, 0|
  0, 0, 0|
  0, 0, 0|
  0, 0, 0|
  0, 0, 0|
  0, 0, 0|
  1, 0, 0|
  1, 0, 0|
  0, 0, 0|
  0, 0, 0|
  0, 0, 0|
  0, 0, 0|
  0, 0, 0|
  0, 0, 0|
  1, 0, 0|$
\newline
$1, 0, 0|
  1, 0, 0|
  2, 0, 0|
  2, 0, 0|
  1, 0, 0|
  1, 0, 0|
  1, 0, 0|
  1, 0, 0|
  1, 0, 0|
  1, 0, 0|
  2, 0, 0|
  2, 0, 0|
  2, 0, 0|
  3, 0, 0|
  3, 0, 0|$
\newline
$2, 0, 0|
  2, 0, 0|
  2, 0, 0|
  2, 0, 0|
  2, 0, 0|
  1, 0, 0|
  1, 0, 0|
  1, 0, 0|
  1, 0, 0|
  2, 0, 0|
  2, 0, 0|
  1, 0, 0|
  1, 0, 0|
  1, 0, 0|
  1, 0, 0|$
\newline
$1, 0, 0|
  1, 0, 0|
  2, 0, 0|
  2, 0, 0|
  1, 0, 0|
  1, 0, 0|
  1, 0, 0|
  1, 0, 0|
  1, 0, 0|
  0, 0, 0|
  0, 0, 0|
  0, 0, 0|
  0, 0, 0|
  1, 0, 0|
  1, 0, 0|$
\newline
$0, 0, 0|
  0, 0, 0|
  0, 0, 0|
  0, 0, 0|
  0, 0, 0|
  0, 0, 0|
  1, 0, 0|
  1, 0, 0|
  1, 0, 0|
  2, 0, 0|
  2, 0, 0|
  1, 0, 0|
  1, 0, 0|
  1, 0, 0|
  1, 0, 0|$
\newline
$1, 0, 0|
  0, 0, 0|
  0, 0, 0|
  0, 0, 0|
  0, 0, 0|
  1, 0, 0|
  1, 0, 0|
  0, 0, 0|
  0, 0, 0|
  0, 0, 0|
  0, 0, 0|
  0, 0, 0|
  0, 0, 0|
  1, 0, 0|
  1, 0, 0|$
\newline
$0, 0, 0|
  0, 0, 0|
  0, 0, 0|
  0, 0, 0|
  0, 0, 0|
  0, 0, 0|
  1, 0, 0|
  1, 0, 0|
  1, 0, 0|
  2, 0, 0|
  2, 0, 0|
  1, 0, 0|
  1, 0, 0|
  1, 0, 0|
  1, 0, 0|$
\newline
$1, 0, 0|
  0, 0, 0|
  0, 0, 0|
  0, 0, 0|
  0, 0, 0|
  1, 0, 0|
  1, 0, 0|
  0, 0, 0|
  0, 0, 0|
  0, 0, 0|
  0, 0, 0|
  0, 0, 0|
  0, 0, 0|
  1, 0, 0|
  1, 0, 0|$
\newline
$0, 0, 0|
  0, 0, 0|
  0, 0, 0|
  0, 0, 0|
  0, 0, 0|
  0, 0, 1|
  0, 0, 0|
  0, 0, 0|
  0, 0, 0|
  1, 0, 0|
  1, 0, 0|
  0, 0, 0|
  0, 0, 0|
  0, 0, 0|
  0, 0, 0|$
\newline
$0, 0, 0|
  0, 0, 0|
  1, 0, 0|
  1, 0, 0|
  1, 0, 0|
  2, 0, 0|
  2, 0, 0|
  1, 0, 0|
  1, 0, 0|
  1, 0, 0|
  1, 0, 0|
  1, 0, 0|
  0, 0, 0|
  0, 0, 0|
  0, 0, 0|$
\newline
$0, 0, 0|
  1, 0, 0|
  1, 0, 0|
  0, 0, 0|
  0, 0, 0|
  0, 0, 0|
  0, 0, 0|
  0, 0, 0|
  0, 0, 0|
  1, 0, 0|
  1, 0, 0|
  0, 0, 0|
  0, 0, 0|
  0, 0, 0|
  0, 0, 0|$
\newline
$0, 0, 0|
  0, 1, 0|
  0, 0, 0|
  0, 0, 0|
  0, 0, 0|
  1, 0, 0|
  1, 0, 0|
  0, 0, 0|
  0, 0, 0|
  0, 0, 0|
  0, 0, 0|
  0, 0, 0|
  0, 0, 0|
  1, 0, 0|
  1, 0, 0|$
\newline
$1, 0, 0|
  2, 0, 0|
  2, 0, 0|
  1, 0, 0|
  1, 0, 0|
  1, 0, 0|
  1, 0, 0|
  1, 0, 0|
  0, 0, 0|
  0, 0, 0|
  0, 0, 0|
  0, 0, 0|
  1, 0, 0|
  1, 0, 0|
  0, 0, 0|$
\newline
$0, 0, 0|
  0, 0, 0|
  0, 0, 0|
  0, 0, 0|
  0, 0, 0|
  1, 0, 0|
  1, 0, 0|
  0, 0, 0|
  0, 0, 0|
  0, 0, 0|
  0, 0, 0|
  0, 0, 0|
  0, 0, 1|
  0, 0, 0|
  0, 0, 0|$
\newline
$0, 0, 0|
  1, 0, 0|
  1, 0, 0|
  0, 0, 0|
  0, 0, 0|
  0, 0, 0|
  0, 0, 0|
  0, 0, 0|
  0, 1, 0|
  0, 0, 0|
  0, 0, 0|
  0, 0, 0|
  1, 0, 0|
  1, 0, 0|
  0, 0, 0|$
\newline
$0, 0, 0|
  0, 0, 0|
  0, 0, 0|
  0, 0, 0|
  0, 0, 1|
  0, 0, 0|
  0, 0, 0|
  0, 1, 0|
  0, 0, 0|
  0, 0, 0|
  0, 0, 1|
  0, 1, 0|
  0, 0, 1)$


\bibliographystyle{amsplain}

\medskip

\noindent
Miwa Iwakura, Chuichiro Hayashi:
Department of Mathematical and Physical Sciences,
Faculty of Science, Japan Women's University,
2-8-1 Mejirodai, Bunkyo-ku, Tokyo, 112-8681, Japan.
hayashic@fc.jwu.ac.jp

\end{document}